\documentclass[12pt]{amsart}
\usepackage{fullpage,url,amssymb}
\usepackage[all]{xy} % for complicated commutative diagrams
\usepackage{mathrsfs} % for \mathscr (script letters)
\usepackage{enumerate}

\DeclareFontEncoding{OT2}{}{} % to enable usage of cyrillic fonts

\usepackage{color}

\newcommand{\defi}[1]{\textsf{#1}} % for defined terms

\newcommand{\scriptH}{\mathscr{H}}

\newcommand{\Aff}{{\mathbb A}}
\newcommand{\C}{{\mathbb C}}
\newcommand{\F}{{\mathbb F}}

\newcommand{\PP}{{\mathbb P}}
\newcommand{\Q}{{\mathbb Q}}
\newcommand{\R}{{\mathbb R}}
\newcommand{\Z}{{\mathbb Z}}
\newcommand{\Qbar}{{\overline{\Q}}}

\newcommand{\kbar}{{\overline{k}}}
\newcommand{\Kbar}{{\overline{K}}}
\newcommand{\Rbar}{{\overline{R}}}

\newcommand{\Fbar}{{\overline{\F}}}

\newcommand{\kk}{{\sf{k}}}

\newcommand{\pp}{{\mathfrak p}}

\newcommand{\calA}{{\mathcal A}}

\newcommand{\calE}{{\mathcal E}}

\newcommand{\calM}{{\mathcal M}}

\newcommand{\calO}{{\mathcal O}}

\newcommand{\calX}{{\mathcal X}}
\newcommand{\calY}{{\mathcal Y}}

\newcommand{\FF}{{\mathscr F}}
\newcommand{\GG}{{\mathscr G}}

\newcommand{\LL}{{\mathscr L}}
\newcommand{\OO}{{\mathscr O}}
\newcommand{\XX}{{\mathscr X}}
\newcommand{\ZZ}{{\mathscr Z}}

\DeclareMathOperator{\spe}{sp}

\DeclareMathOperator{\jumping}{jumping}
\DeclareMathOperator{\Coh}{Coh}
\DeclareMathOperator{\id}{id}
\DeclareMathOperator{\Fil}{Fil}
\DeclareMathOperator{\Spf}{Spf}
\DeclareMathOperator{\ev}{ev}
\DeclareMathOperator{\Betti}{Betti}

\DeclareMathOperator{\cl}{cl}

\DeclareMathOperator{\coker}{coker}
\DeclareMathOperator{\rk}{rk}
\DeclareMathOperator{\Char}{char}

\DeclareMathOperator{\im}{im}

\DeclareMathOperator{\End}{End}

\DeclareMathOperator{\Hom}{Hom}

\DeclareMathOperator{\Gal}{Gal}

\DeclareMathOperator{\Res}{Res}

\DeclareMathOperator{\Br}{Br}

\DeclareMathOperator{\Pic}{Pic}
\DeclareMathOperator{\NS}{NS}

\DeclareMathOperator{\PIC}{\bf Pic}
\DeclareMathOperator{\Spec}{Spec}

\DeclareMathOperator{\Frac}{Frac}

\newcommand{\cris}{{\operatorname{cris}}}
\newcommand{\dR}{{\operatorname{dR}}}

\newcommand{\alg}{{\operatorname{alg}}}

\newcommand{\an}{{\operatorname{an}}}

\newcommand{\et}{{\operatorname{\textup{\'et}}}}

\newcommand{\red}{{\operatorname{red}}}

\newcommand{\HH}{{\operatorname{H}}}
\newcommand{\RR}{{\operatorname{R}}}

\newcommand{\injects}{\hookrightarrow}
\newcommand{\isom}{\simeq}

\newcommand{\intersect}{\cap} % binary intersection
\newcommand{\Union}{\bigcup} % union of a collection
\newcommand{\union}{\cup} % binary union
\newcommand{\tensor}{\otimes}
\newcommand{\directsum}{\oplus} % binary direct sum
\newcommand{\isomto}{\overset{\sim}{\rightarrow}}
\newcommand{\isomfrom}{\overset{\sim}{\leftarrow}}

\numberwithin{equation}{section}
\newtheorem{theorem}[equation]{Theorem}
\newtheorem{lemma}[equation]{Lemma}
\newtheorem{corollary}[equation]{Corollary}
\newtheorem{proposition}[equation]{Proposition}
\newtheorem{conjecture}[equation]{Conjecture}

\theoremstyle{definition}
\newtheorem{definition}[equation]{Definition}
\newtheorem{setup}[equation]{Setup}
\newtheorem{question}[equation]{Question}
\newtheorem{example}[equation]{Example}

\newtheorem{claim}[equation]{Claim}

\theoremstyle{remark}
\newtheorem{remark}[equation]{Remark}

\usepackage[
	backref,
	pdfauthor={Bjorn Poonen}, % add other authors
]{hyperref}
\usepackage[alphabetic,backrefs,lite]{amsrefs} % for bibliography

\begin{document}

\title{N\'{e}ron-Severi groups under specialization}
\author{Davesh Maulik}
\address{Department of Mathematics, Massachusetts Institute of Technology, Cambridge, MA 02139-4307, USA}
\email{dmaulik@math.mit.edu}
\author{Bjorn Poonen}
\address{Department of Mathematics, Massachusetts Institute of Technology, Cambridge, MA 02139-4307, USA}
\email{poonen@math.mit.edu}
\urladdr{http://math.mit.edu/~poonen/}
% \author{Claire Voisin}
% \address{Institut de Math\'{e}matiques de Jussieu, 175 rue du Chevaleret, 75013 Paris, France}
% \email{voisin@math.jussieu.fr}
\thanks{D.M.\ is supported by a Clay Research Fellowship.  B.P. is partially supported by NSF grant DMS-0841321.}
\date{October 30, 2010}

\begin{abstract}
Andr\'e used Hodge-theoretic methods to 
show that in a smooth proper family $\calX \to B$
of varieties over an algebraically closed field $k$ of characteristic~$0$,
there exists a closed fiber having
the same Picard number as the geometric generic fiber,
even if $k$ is countable.
We give a completely different approach to Andr\'e's theorem, 
which also proves the following refinement:
in a family of varieties with good reduction at $p$,
the locus on the base where the Picard number jumps 
is nowhere $p$-adically dense.
Our proof uses the ``$p$-adic Lefschetz $(1,1)$ theorem''
of Berthelot and Ogus, combined with an analysis of $p$-adic power series.
We prove analogous statements for cycles of higher codimension,
assuming a $p$-adic analogue of the variational Hodge conjecture,
and prove that this analogue implies the usual variational Hodge conjecture.
Applications are given to abelian schemes and to proper families
of projective varieties.
\end{abstract}

\maketitle

%****************************************************************************
\section{Introduction}\label{S:introduction}

\subsection{The jumping locus}

For a smooth proper variety $X$ over an algebraically closed field,
let $\NS X$ be its N\'{e}ron-Severi group,
and let $\rho(X)$ be the rank of $\NS X$.
(A \defi{variety} is a separated scheme of finite type over a field, 
possibly non-reduced or reducible.
See Sections \ref{S:notation} and~\ref{S:NS groups}
for further definitions and basic facts.)

Now suppose that we have a smooth proper morphism $\calX \to B$,
where $B$ is an irreducible variety over
an algebraically closed field $k$ of characteristic $0$.
If $b \in B(k)$, then choices lead to an injection
of the N\'{e}ron-Severi group $\NS \calX_{\bar{\eta}}$
of the geometric generic fiber into
the N\'{e}ron-Severi group $\NS \calX_b$ of the fiber above $b$,
so $\rho(\calX_b) \ge \rho(\calX_{\bar{\eta}})$: 
see Proposition~\ref{P:construction of specialization 2}.
The \defi{jumping locus}
\[
	B(k)_{\jumping} :=
	\{ b \in B(k) : \rho(\calX_b) > \rho(\calX_{\bar{\eta}}) \}
\]
is a countable union of lower-dimensional subvarieties of $B$.
If $k$ is uncountable, it follows that $B(k)_{\jumping} \ne B(k)$.

This article concerns the general case, in which $k$ may be countable.
Our goal is to present a $p$-adic proof of
the following theorem, first proved by 
Y.~Andr\'e~\cite{Andre1996}*{Th\'eor\`eme 5.2(3)}:

\begin{theorem}
\label{T:global}
Let $k$ be an algebraically closed field of characteristic $0$.
Let $B$ be an irreducible variety over $k$.
Let $\calX \to B$ be a smooth proper morphism.
Then there exists $b \in B(k)$ such that
$\rho(\calX_b) = \rho(\calX_{\bar{\eta}})$.
\end{theorem}

\begin{remark}
In fact, Y.~Andr\'e's result is more general,
stated in terms of variation of the motivic Galois group 
in the context of his theory of ``motivated cycles''.
On the other hand, our techniques, which are completely different,
give new information about the jumping locus.

Special cases were proved earlier by T.~Shioda~\cite{Shioda1981}
and most notably T.~Terasoma~\cite{Terasoma1985}.
The arguments of Terasoma and Andr\'e involve, among other ingredients,
an application of a version of Hilbert irreducibility
for infinite algebraic extensions associated to $\ell$-adic representations.
We will say more about their methods and their relationship with ours
in Section~\ref{S:complex outline}.
\end{remark}

\begin{remark}
\label{R:NS versus its rank}
The condition $\rho(\calX_b) = \rho(\calX_{\bar{\eta}})$
is equivalent to the condition
that the specialization map $\NS \calX_{\bar{\eta}} \to \NS \calX_b$ 
is an isomorphism:
see Proposition~\ref{P:construction of specialization 2}.
\end{remark}

\begin{remark}
Theorem~\ref{T:global} can be trivially extended 
to an arbitrary ground field $k$ of characteristic $0$, 
to assert the existence of a closed point $b \in B$
such that the {\em geometric} Picard number of $\calX_b$ 
equals $\rho(\calX_{\bar{\eta}})$.
Similarly, one could relax the assumption on $B$ and allow it to be any
irreducible scheme of finite type over $k$.
\end{remark}

\begin{remark}
For explicit examples of specializations where the Picard number 
does not jump, see \cite{Shioda1981} and \cite{VanLuijk2007}.
\end{remark}

\subsection{The $p$-adic approach}
\label{S:p-adic outline}

For our proof of Theorem~\ref{T:global},
we embed a suitable finitely generated subfield of $k$ in a $p$-adic field
(see Section~\ref{S:proof of local and global theorems})
and apply Theorem~\ref{T:local} below,
which states that for a family of varieties with good reduction,
the jumping locus is {\em nowhere $p$-adically dense}.

\begin{setup}
\label{A:STAR}
Let $K$ be a field that is complete with respect to a discrete valuation,
and let $\kk$ be the residue field.
For any valued field $L$, let $\calO_L$ denote its valuation ring.
Assume that $\Char K = 0$ and $\Char \kk = p > 0$,
and that $\kk$ is perfect.
Let $C$ be the completion of an algebraic closure of $K$;
then $C$ also is algebraically closed \cite{Kurschak1913}*{\S46}.  
Let $B$ be an irreducible separated finite-type $\calO_K$-scheme,
and let $f\colon \calX \to B$ be a smooth proper morphism.
\end{setup}

\begin{theorem}
\label{T:local}
Assume Setup~\ref{A:STAR}.
For $b \in B(\calO_C) \subseteq B(C)$,
let $\calX_b$ be the $C$-variety above $b$.
Then the set
\[
	B(\calO_C)_{\jumping} := \{b \in B(\calO_C) :
	\rho(\calX_b)>\rho(\calX_{\bar{\eta}}) \}
\]
is nowhere dense in $B(\calO_C)$ for the analytic topology.
\end{theorem}

To prove Theorem~\ref{T:local},
we apply a ``$p$-adic Lefschetz $(1,1)$ theorem''
of P.~Berthelot and A.~Ogus~\cite{Berthelot-Ogus1983}*{Theorem~3.8}
to obtain a down-to-earth
local analytic description (Lemma~\ref{L:isocrystal output})
of the jumping locus in $B(\calO_C)$.
This eventually reduces the problem
to a peculiar statement (Proposition~\ref{P:evaluation is injective})
about linear independence of values
of linearly independent $p$-adic power series.

\begin{remark}
\label{R:archimedean}
It is well known (cf.~\cite{Bosch-Lutkebohmert-Raynaud1990}*{p.~235})
that the archimedean analogue of Theorem~\ref{T:local} is false.
For example, let $B$ be an irreducible $\C$-variety,
let $\calE \to B$ be a family of elliptic curves
such that the $j$-invariant map $j \colon B \to \Aff^1$ is dominant,
and let $\calX = \calE \times_B \calE$.
For an elliptic curve $E$ over an algebraically closed field,
$\rho(E \times E) = 2 + \rk \End E$
(cf.~the Rosati involution comment in the proof of
Proposition~\ref{P:specialization of endomorphism rings}).
So $B(\C)_{\jumping}$ is the set of CM points in $B(\C)$.
In the $j$-line, the set of CM points is
the image of $\{z \in \C : \im(z)>0 \text{ and } [\Q(z):\Q]=2\}$
under the usual analytic uniformization by the upper half plane.
This image is dense in $\Aff^1(\C)$, so its preimage under $j$ is
dense in $B(\C)$.
\end{remark}

\begin{remark}
Remark~\ref{R:archimedean} is a particular case
of a general topological density theorem for the jumping locus
(cf. \cite{Voisin2003vol2}*{5.3.4}):
In the setting of Theorem~\ref{T:global} over $k=\C$,
if the jumping locus in $B$ has a component that is
reduced and of the ``expected codimension'' $h^{2,0}$,
then the jumping locus is topologically dense in $B(\C)$.
This fact ultimately relies on the topological density of $\Q$ inside $\R$.
\end{remark}

\begin{remark}
We can give a heuristic explanation of the difference
between $\C$ and a field like $C=\C_p$.
Namely, $[\C:\R]=2$, but the analogous $p$-adic quantity $[\C_p:\Q_p]$
is infinite (in fact, equal to $2^{\aleph_0}$ \cite{Lampert1986}).
So a subvariety in $B(\C_p)$ of positive codimension can be thought of
as having infinite $\Q_p$-codimension.
This makes it less surprising that a countable union of such subvarieties
could be nowhere $p$-adically dense.
\end{remark}

\begin{remark}
G.~Yamashita, in response to an earlier version of this article,
has generalized the $p$-adic Lefschetz $(1,1)$ theorem
from the smooth case to the semistable case,
and has used our method to extend Theorem~\ref{T:local}
to the case where $\calX \to B$ is semistable \cite{Yamashita-preprint}.
\end{remark}

\begin{remark}
If $\calX \to B$ is as in Remark~\ref{R:archimedean},
but over an algebraic closure $k$ of a finite field $\F_p$,
then again we have $\rho(\calX_{\bar{\eta}}) = 3$,
but now $\rho(\calX_b) \ge 4$ for all $b \in B(k)$
since every elliptic curve $E$ over $k$
has endomorphism ring larger than $\Z$.
Thus the characteristic~$p$ analogue of Theorem~\ref{T:global} fails.
On the other hand, it seems likely that it holds
for any algebraically closed field $k$
that is not algebraic over a finite field.
\end{remark}

\subsection{Applications to abelian varieties}

J.-P.~Serre~\cite{SerreOeuvres-IV}*{pp.~1--17} 
and R.~Noot~\cite{Noot1995}*{Corollary~1.5} 
used something like Terasoma's method,
combined with G.~Faltings' proof of the Tate conjecture for 
homomorphisms between abelian varieties, to prove that 
in a family of abelian varieties over a finitely generated field
of characteristic $0$,
there exists a geometric closed fiber whose endomorphism ring
equals that of the geometric generic fiber.
Independently at around the same time, 
D.~Masser~\cite{Masser1996} used methods of transcendence theory
to give a different proof, one that gives quantitative estimates
of the number of fibers where the endomorphism ring jumps.

Theorem~\ref{T:global} 
reproves the existence result without Faltings' work or transcendence theory,
and Theorem~\ref{T:local} strengthens this by showing 
that in the $p$-adic setting,
the corresponding jumping locus is nowhere $p$-adically dense 
in the good reduction locus:

\begin{proposition}
\label{P:specialization of endomorphism rings}
Assume Setup~\ref{A:STAR},
and assume moreover that $\calX \to B$ is an abelian scheme.
Then
\[
	\{b \in B(\calO_C): \End \calX_{\bar{\eta}} \injects \End \calX_b 
		\text{ is not an isomorphism}\}
\]
is nowhere dense in $B(\calO_C)$ for the analytic topology.
\end{proposition}

\begin{proof}
Choose a polarization on $\calX_\eta$, 
and replace $B$ by a dense open subvariety
to assume that it extends to a polarization of $\calX \to B$.
For a polarized abelian variety $A$ over an algebraically closed field, 
the group $(\NS A)_\Q$ is isomorphic to the subspace of $(\End A)_\Q$
fixed by the Rosati involution: 
see~\cite{MumfordAV1970}*{p.~190}, for instance.
(The subscript $\Q$ denotes $\tensor \Q$.)
This implies
\[
	\rho(A \times A) = 2 \rho(A) + \dim (\End A)_\Q.
\]
If in a family, $\dim (\End A)_\Q$ jumps, then so will $\rho(A \times A)$;
conversely, if it does not, then neither does $\rho(A)$, so $\rho(A \times A)$
also does not jump.
Thus the $(\End A)_\Q$ jumping locus for $\calX \to B$
equals the Picard number jumping locus for $\calX \times_B \calX \to B$.
Apply Theorem~\ref{T:local} to $\calX \times_B \calX$.
Finally, 
$(\End \calX_{\bar{\eta}})_\Q \injects (\End \calX_b)_\Q$
is an isomorphism if and only if 
$\End \calX_{\bar{\eta}} \injects \End \calX_b$
is an isomorphism,
as one sees by considering the action on torsion points
(this uses characteristic~$0$).
\end{proof}

\begin{remark}
Theorem~1.7 of~\cite{Noot1995} states that for any algebraic group $G$
arising as a Mumford-Tate group of a complex abelian variety,
there exists an abelian variety $A$ over a number field $F$
such that the Mumford-Tate group of $A$ equals $G$
and such that moreover the Mumford-Tate conjecture holds;
i.e., the action of $\Gal(\overline{F}/F)$ on a Tate module $T_\ell A'$
gives an {\em open} subgroup in $G(\Q_\ell)$.
A specialization result for the Mumford-Tate group follows easily
from \cite{Andre1996}*{Th\'eor\`eme 5.2(3)} too.

It would be natural to conjecture a ``nowhere dense'' analogue,
i.e., that the locus in a family of abelian varieties
where the dimension of the Mumford-Tate group drops
is nowhere $p$-adically dense in the good reduction locus.
But we know how to prove this only if we assume Conjecture~\ref{C:Emerton}
from Section~\ref{S:higher codimension}.
\end{remark}

A proof similar to that of 
Proposition~\ref{P:specialization of endomorphism rings}
yields another application of Theorem~\ref{T:global}:
\begin{proposition} 
\label{P:isogeny} 
Let $k$ be an algebraically closed field of characteristic~$0$.
Let $A$ be an abelian variety defined over $k$.
Let $B$ be an irreducible $k$-variety.
Let $\calX \to B$ be an abelian scheme such that
$\calX_b$ is isogenous to $A$ for all $b \in B(k)$.
Then $\calX_{\bar{\eta}}$ is isogenous to $A_{\bar{\eta}}:=A \times_k \bar{\eta}$.
\end{proposition}

\begin{proof}[Sketch of proof]
Let $A \sim \prod_{i=1}^r A_i^{n_i}$ be a decomposition of $A$ up to isogeny
into simple factors.
Applying Theorem~\ref{T:global} to $\calX \times A_i \to B$ 
shows that the multiplicity of $(A_i)_{\bar{\eta}}$ 
in the decomposition of $\calX_{\bar{\eta}}$
equals $n_i$.
Since the relative dimension of $\calX \to B$ equals $\dim A$,
this accounts for all simple factors of $\calX_{\bar{\eta}}$.
\end{proof}

\begin{remark}
At least when $\calX \to B$ is projective 
(which is automatic if $B$ is normal~\cite{Faltings-Chai1990}*{1.10(a)}),
the conclusion of Proposition~\ref{P:isogeny} 
implies that $\calX \to B$ becomes constant after a finite \'etale
base change $B' \to B$.
This can be proved as follows.
The kernel of an isogeny $A_{\bar{\eta}} \to \calX_{\bar{\eta}}$
is the base extension of a finite group scheme $G$ over $k$,
since $k$ is algebraically closed of characteristic~$0$.
Replacing $A$ by $A/G$, we may assume that $A_{\bar{\eta}} \isom \calX_{\bar{\eta}}$.
Projectivity of $\calX \to B$ yields a polarization on $\calX$,
and the corresponding polarization on $A_{\bar{\eta}}$
comes from a polarization defined over $k$.
Choose $\ell \ge 3$, and replace $B$ by a finite \'etale cover
such that $\calX[\ell] \isom (\Z/\ell\Z)^{2g}_B$.
This lets us choose level-$\ell$ structures so that
$A_{\bar{\eta}} \to \calX_{\bar{\eta}}$ becomes an isomorphism
of polarized abelian varieties with level-$\ell$ structure.
Let $\calM$ is the moduli scheme of 
polarized abelian varieties with level-$\ell$ structure.
Then $\calX$ give rise to a morphism $B \to \calM$
whose restriction to $\bar{\eta}$ agrees with a constant map.
Since $\calM$ is separated, $B \to \calM$ itself is constant.
   % This remark can fail for non-projective $\calX \to B$, at least for non-reduced $B$ like $B=\Spec k[e]/(e^2)$.
   % We don't know if it holds without the projectivity assumption.
\end{remark}

\begin{remark}
Under the appropriate hypotheses on $k$ and $\calX \to B$,
Theorem~\ref{T:local}
proves the analogous strengthening of Proposition~\ref{P:isogeny}.
\end{remark}

\subsection{Outline of the article}

After introducing some notation in Section~\ref{S:notation},
we review some standard facts about N\'{e}ron-Severi groups
and specialization maps in Section~\ref{S:NS groups}.
The next three sections prove Theorem~\ref{T:local} and use
it to prove Theorem~\ref{T:global}:
Section~\ref{S:convergent isocrystals}
discusses some basic properties of crystalline cohomology
and convergent isocrystals,
and applies them to give a local description of the jumping locus;
Section~\ref{S:unions of zero loci} proves the key $p$-adic power series
proposition to be applied to understand this local description.
Section~\ref{S:proof of local and global theorems} completes
the proofs of Theorems \ref{T:local} and~\ref{T:global}.

Section~\ref{S:projective vs proper} gives an application
of Theorem~\ref{T:global}:
if all closed fibers in a smooth proper family
are projective, then there exists a dense open subvariety of the base
over which the family is projective,
assuming that the base is a variety in characteristic~$0$.
Section \ref{S:complex outline}, which
uses only \'{e}tale and Betti cohomology, and some Hodge theory,
sketches Andr\'e's approach to Theorem~\ref{T:global},
and compares the information it provides on the jumping locus
to what is obtained from the $p$-adic approach.

Finally, Section~\ref{S:higher codimension}
explains conditional generalizations of our results
to cycles of higher codimension.
The generalization of Theorem~\ref{T:local}
is proved assuming
a $p$-adic version of the variational Hodge conjecture
(Conjecture~\ref{C:Emerton}).
We also prove that the $p$-adic variational Hodge conjecture
implies the classical variational Hodge conjecture.

%****************************************************************************
\section{Notation}
\label{S:notation}

If $A$ is a commutative domain, let $\Frac(A)$ denote its fraction field.
If $A \to B$ is a ring homomorphism,
and $M$ is an $A$-module,
let $M_B$ denote the $B$-module $M \tensor_A B$.
If $k$ is a field, then $\kbar$ denotes an algebraic closure,
chosen consistently whenever possible.
Given a prime number $p$,
let $\Z_p$ be the ring of $p$-adic integers,
let $\Q_p=\Frac(\Z_p)$,
choose algebraic closures $\Qbar \subseteq \Qbar_p$,
and let $\C_p$ denote the completion of $\Qbar_p$.

For any $S$-schemes $X$ and $T$, let $X_T$ be the $T$-scheme $X \times_S T$.
For a commutative ring $R$, we may write $R$ as an abbreviation for $\Spec R$.
If $B$ is an irreducible scheme,
let $\eta$ denote its generic point.
If $b \in B$, let $\kappa(b)$ be its residue field
and let $\bar{b}=\Spec \overline{\kappa(b)}$.
For example, if $\calX \to B$ is a morphism,
then $\calX_{\bar{\eta}}$ is called the \defi{geometric generic fiber}.
Also let $\kappa(B)$ be the function field $\kappa(\eta)$.
If $B$ is a variety over a field $F$, let $|B|$ be the set of closed
points of $B$; also choose an algebraic closure $\overline{F}$ 
and for all $b \in |B|$
view $\overline{\kappa(b)}$ as a subfield of $\overline{F}$.
If $X$ is a variety over a field equipped with an embedding in $\C$,
then $X^\an$ denotes the associated complex analytic space.

If $\XX$ is a complex analytic space and $i$ is a
nonnegative integer, then we have the \defi{Betti cohomology} $\HH^i(\XX,F)$
for any field $F$.
If $X$ is a variety over a field $k$, and $i$ and $j$ are integers
with $i \ge 0$, and $\ell$ is a prime not divisible by the characteristic
of $k$,
then we have the \defi{\'{e}tale cohomology} $\HH^i_\et(X_\kbar,\Q_\ell(j))$,
which is equipped with a $\Gal(\kbar/k)$-action
(replace $\kbar$ by a separable closure if $k$ is not perfect).

%****************************************************************************
\section{Basic facts on N\'{e}ron-Severi groups}
\label{S:NS groups}

\subsection{Picard groups and N\'{e}ron-Severi groups}

For a scheme or formal scheme $X$, let $\Pic X$ be its Picard group.
If $X$ is a smooth proper variety over an algebraically closed field,
let $\Pic^0 X$ be the subgroup consisting
of isomorphism classes of line bundles
algebraically equivalent to $0$ (i.e., to $\OO_X$),
and define the \defi{N\'{e}ron-Severi group} $\NS X := \Pic X / \Pic^0 X$.
The abelian group $\NS X$ is
finitely generated \cite{Neron1952}*{p.~145,~Th\'{e}or\`{e}me~2}
(see \cite{SGA6}*{XIII.5.1} for another proof),
and its rank is called the \defi{Picard number} $\rho(X)$.

\begin{proposition}
\label{P:algebraically closed base change}
If $k \subseteq k'$ are algebraically closed fields,
and $X$ is a smooth proper $k$-variety, then the natural homomorphism
$\NS X \to \NS X_{k'}$ is an isomorphism.
\end{proposition}

\begin{proof}
The Picard scheme $\PIC_{X/k}$ is a group scheme that is locally of finite type
over $k$ (this holds more generally for any
proper scheme over a field: see \cite{Murre1964}*{II.15},
which uses \cite{Oort1962}).
Then $\Pic^0 X$ is the set of $k$-points of the
identity component of $\PIC_{X/k}$ \cite{Kleiman2005}*{9.5.10}.
So $\NS X$ is the group of components of $\PIC_{X/k}$.
Thus $\NS X$ is unchanged by algebraically closed base extension.
\end{proof}

\begin{remark}
\label{R:Nakai-Moishezon}
The Nakai-Moishezon criterion \cite{Debarre2001}*{Theorem~1.21}
implies that ampleness of a Cartier divisor
on a proper scheme $X$ over any field $K$
depends only on its class in $\NS X_L$
for any algebraically closed field $L$ containing $K$.
\end{remark}

\subsection{Specialization of N\'{e}ron-Severi groups}
\label{S:specialization of NS}

\begin{proposition}[cf.~\cite{SGA6}*{X~App~7}]
\label{P:construction of specialization}
Let $R$ be a discrete valuation ring
with fraction field $K$ and residue field $k$.
Fix an algebraic closure $\Kbar$ of $K$.
Choose a nonzero prime ideal $\pp$
of the integral closure $\Rbar$ of $R$ in $\Kbar$,
so $\kbar:=\Rbar/\pp$ is an algebraic closure of $k$.
Let $X$ be a smooth proper $R$-scheme.
Then there is a natural homomorphism
\[
	\spe_{\Kbar,\kbar}\colon \NS X_\Kbar \to \NS X_\kbar.
\]
depending only on the choices above.
Moreover, if $\spe_{\Kbar,\kbar}$ maps a class $[\LL]$ to an ample class,
then $\LL$ is ample.
\end{proposition}

\begin{proof}
As in \cite{SGA6}*{X~App~7.8} 
or the proof of \cite{Bosch-Lutkebohmert-Raynaud1990}*{\S8.4,~Theorem~3}, 
we have
\begin{equation}
\label{E:Pic specialization}
	\Pic X_K \isomfrom \Pic X \to \Pic X_k.
\end{equation}
If $\LL$ is a line bundle on $X_K$
whose image in $\Pic X_k$ is ample,
then the corresponding line bundle on $X$ is ample relative to $\Spec R$
by \cite{EGA-III.I}*{4.7.1},
so $\LL$ is ample too.

For each finite extension $L$ of $K$ in $\Kbar$,
the integral closure $R_L$ of $R$ in $L$ is a Dedekind ring
by the Krull-Akizuki theorem~\cite{BourbakiCommutativeAlgebra}*{VII.2.\S5,~Proposition~5},
and localizing at $\pp \intersect R_L$
gives a discrete valuation ring $R_L'$.
Take the direct limit over $L$ of the analogue of \eqref{E:Pic specialization}
for $R_L'$ to get
$\Pic X_\Kbar \to \Pic X_\kbar$ (cf.~\cite{SGA6}*{X~App~7.13.3}).

This induces $\NS X_\Kbar \to \NS X_\kbar$ (cf.~\cite{SGA6}*{X~App~7.12.1});
a sketch of an alternative argument for this is as follows:
First we can pass from $R$ to its completion to reduce to the case
that $R$ is excellent.  
It suffices to show the following 
(after replacing $R$ and $K$ by finite extensions):
Given a smooth proper geometrically connected $K$-curve $C_K$
and a line bundle $\LL_K$ on $X_K \times C_K$,
any two fibers above points in $C_K(K)$
specialize to algebraically equivalent line bundles on $X_k$.
By Lipman's resolution of singularities 
for $2$-dimensional excellent schemes~\cite{Lipman1978},
$C_K$ extends to a regular proper flat $R$-scheme $C$,
and $C_k$ is geometrically connected by 
Stein factorization (cf.~\cite{EGA-III.I}*{4.3.12}).
The two specialized line bundles are fibers above points of $C(k)$
of an extension of $\LL_K$ to the regular scheme $X \times_R C$,
so they are algebraically equivalent.

The ampleness claim follows from
Remark~\ref{R:Nakai-Moishezon} and the statement for $\Pic$
already discussed.
\end{proof}

\begin{remark}
In Proposition~\ref{P:construction of specialization},
if $R$ is complete, or more generally henselian,
then there is only one choice of $\pp$.
\end{remark}

\begin{proposition}
\label{P:construction of specialization 2}
Let $B$ be a noetherian scheme.
Let $s,t \in B$ be such that $s$ is a specialization of $t$
(i.e., $s$ is in the closure of $\{t\}$).
Let $p=\Char \kappa(s)$.
Let $\calX \to B$ be a smooth proper morphism.
Then it is possible to choose a homomorphism
\[
	\spe_{\bar{t},\bar{s}}\colon
	\NS \calX_{\bar{t}} \to \NS \calX_{\bar{s}}
\]
with the following properties:
\begin{enumerate}[\upshape (a)]
\item\label{I:specialization injectivity}
If $p=0$, then $\spe_{\bar{t},\bar{s}}$ is injective
and $\coker(\spe_{\bar{t},\bar{s}})$ is torsion-free.
\item\label{I:p>0}
If $p>0$, then \eqref{I:specialization injectivity}
holds after tensoring with $\Z[1/p]$.
\item\label{I:rho inequality}
In all cases, $\rho(\calX_{\bar{s}}) \ge \rho(\calX_{\bar{t}})$.
\item\label{I:specializing ampleness}
If $\spe_{\bar{t},\bar{s}}$ maps a class $[\LL]$ to an ample class,
then $\LL$ is ample.
\end{enumerate}
\end{proposition}

\begin{proof}
A construction of $\spe_{\bar{t},\bar{s}}$
is explained at the beginning of \cite{SGA6}*{X~App~7.17}:
the idea is to choose a discrete valuation ring $R$
with a morphism $\Spec R = \{s',t'\} \to B$
mapping $s'$ to $s$ and $t'$ to $t$,
to obtain
\[
	\NS \calX_{\bar{t}}
	\isomto \NS \calX_{\bar{t}'}
	\stackrel{\spe_{\bar{t}',\bar{s}'}}\longrightarrow
	\NS \calX_{\bar{s}'}
	\isomfrom \NS \calX_{\bar{s}},
\]
with the outer isomorphisms coming from
Proposition~\ref{P:algebraically closed base change}.

For any prime $\ell \ne p$, there is a commutative diagram
\begin{equation}
\label{E:NS and H^2}
\xymatrix{
	\NS \calX_{\bar{t}} \tensor \Z_\ell \ar@{^{(}->}[r] \ar[d]_{\spe_{\bar{t},\bar{s}}} & \HH^2_\et(\calX_{\bar{t}},\Z_\ell(1)) \ar@{=}[d] \\
	\NS \calX_{\bar{s}} \tensor \Z_\ell \ar@{^{(}->}[r] & \HH^2_\et(\calX_{\bar{s}},\Z_\ell(1)) \\
}
\end{equation}
(cf.~\cite{SGA6}*{7.13.10}: there everything is tensored with $\Q$,
but the explanation shows that in our setting we need only tensor with $\Z[1/(i-1)!]$
with $i=1$).
This proves the injectivity
in \eqref{I:specialization injectivity} and~\eqref{I:p>0}.
By~\eqref{E:NS and H^2},
$\coker(\spe_{\bar{t},\bar{s}}) \tensor \Z_\ell$ is contained in
$\coker\left(\NS \calX_{\bar{t}} \tensor \Z_\ell \to \HH^2_\et(\calX_{\bar{t}},\Z_\ell(1)) \right)$.
Using the Kummer sequence,
one shows~\cite{MilneEtaleCohomology1980}*{V.3.29(d)}
that the latter is 
$T_\ell \Br \calX_{\bar{t}} := \varprojlim_n (\Br \calX_{\bar{t}})[\ell^n]$,
which is automatically torsion-free;
this proves the torsion-freeness 
in \eqref{I:specialization injectivity} and~\eqref{I:p>0}.
Finally, \eqref{I:rho inequality} follows from the earlier parts,
and \eqref{I:specializing ampleness} follows from the corresponding
part of Proposition~\ref{P:construction of specialization}.
\end{proof}

\begin{proposition}
\label{P:rho at least n}
Let $B$ be a noetherian scheme.
For a smooth proper morphism $\calX \to B$
and a nonnegative integer $n$,
define
\[
	B_{\ge n} := \{b \in B : \rho(\calX_{\bar{b}}) \ge n \}.
\]
\begin{enumerate}[\upshape (a)]
\item\label{I:countable union}
The set $B_{\ge n}$ is a countable union of Zariski closed subsets of $B$.
\item\label{I:base change of countable union} 
If we base change by a morphism $\iota \colon B' \to B$ 
of noetherian schemes,
then $B'_{\ge n} = \iota^{-1}(B_{\ge n})$.
\end{enumerate}
\end{proposition}

\begin{proof}
Proposition~\ref{P:algebraically closed base change} 
proves~\eqref{I:base change of countable union}.

Now we prove~\eqref{I:countable union}.
Proposition~\ref{P:construction of specialization 2}\eqref{I:rho inequality}
says that $B_{\ge n}$ contains the closure of any point in $B_{\ge n}$.
So if $B = \Spec A$ for some finitely generated $\Z$-algebra $A$,
then $B_{\ge n}$ is the (countable) union over $b \in B_{\ge n}$
of the closure of $\{b\}$.
Combining this with \eqref{I:base change of countable union} proves 
\eqref{I:countable union} for any noetherian affine scheme.
Finally, if $B$ is any noetherian scheme,
write $B = \Union_{i=1}^n B_i$ with $B_i$ affine,
let $C_i$ be the union of the closures in $B$ of the generic points
of all the irreducible components of the closed subsets of $B_i$
appearing in the countable union for $(B_i)_{\ge n}$,
and let $C = \Union_{i=1}^n C_i$.
Then $B_{\ge n} = \Union_{i=1}^n (B_i)_{\ge n} \subseteq C$
and the opposite inclusion follows using
Proposition~\ref{P:construction of specialization 2}\eqref{I:rho inequality}
again as above.
\end{proof}

\begin{corollary}
\label{C:jumping locus is countable union}
Let $k \subseteq k'$ be algebraically closed fields.
Let $B$ be an irreducible $k$-variety.
For a smooth proper morphism $\calX \to B$, the jumping locus
\[
	B(k')_{\jumping}:= \{b \in B(k') :
		\rho(\calX_b) > \rho(\calX_{\bar{\eta}}) \}
\]
is the union of $Z(k')$ where $Z$ ranges over a countable collection
of closed $k$-subvarieties of $B$.
\end{corollary}

\begin{proof}
Proposition~\ref{P:rho at least n}\eqref{I:countable union} 
yields subvarieties $Z$
for the case $k'=k$.
The same subvarieties work for larger $k'$
by Proposition~\ref{P:rho at least n}\eqref{I:base change of countable union}.
\end{proof}

\subsection{Pathological behavior in positive characteristic}

The material in this section is not needed elsewhere in this article.
Let $R$ be a discrete valuation ring, and define $K,k,\Kbar,\kbar$
as in Section~\ref{S:specialization of NS}.
The two examples below show that $\spe_{\Kbar,\kbar}$ is not always injective.

\begin{example}
There exist $R$ of equal characteristic $2$ and 
a smooth proper morphism $X \to \Spec R$
such that $X_\Kbar$ and $X_\kbar$
are Enriques surfaces of type $\Z/2\Z$ and $\alpha_2$,
respectively \cite{Bombieri-Mumford1976}*{p.~222}.
(The type refers to the isomorphism class of the scheme $\PIC^\tau$
parametrizing line bundles numerically equivalent to $0$.)
In this case $\NS X_\Kbar \to \NS X_\kbar$ has a nontrivial kernel,
generated by the canonical class of $X_\Kbar$,
an element of order $2$.
\end{example}

\begin{example}
There exist $R$ of mixed characteristic $(0,2)$ and 
a smooth proper morphism $X \to \Spec R$
such that $X_\Kbar$ and $X_\kbar$
are Enriques surfaces of type $\Z/2\Z$ and $\mu_2$,
respectively \cite{Lang1983}*{Theorem~1.3},
so again we have a nontrivial kernel.
\end{example}

Next, we give an example showing that $\coker(\spe_{\Kbar,\kbar})$
is not always torsion-free.

\begin{example}
Let $\calO$ be the maximal order of an imaginary quadratic field
in which $p$ splits.
Let $\calO'$ be the order of conductor $p$ in $\calO$.
Over a finite extension $R$ of $\Z_p$,
there exists a $p$-isogeny $\psi \colon E \to E'$
between elliptic curves over $R$
such that $\End E_\Kbar \isom \calO$
and $\End E'_\Kbar \isom \calO'$.
Since $p$ splits, $E$ has good ordinary reduction
and $\End E_\kbar \isom \calO$.
But $\psi$ must reduce to either Frobenius or Verschiebung,
so $\End E'_\kbar \isom \calO$ too.
Using that $\coker\left(\End E'_\Kbar \to \End E'_\kbar\right)$
is of order $p$,
one can show that the cokernel of
$\NS((E' \times E')_\Kbar) \to \NS((E' \times E')_\kbar)$
contains nonzero elements of order $p$.
\end{example}

%****************************************************************************
\section{Convergent isocrystals and de Rham cohomology}
\label{S:convergent isocrystals}

We now begin work toward the $p$-adic proof of Theorem~\ref{T:global}.

\subsection{Goal of this section}

\begin{definition}
\label{D:polydisk}
Assume Setup~\ref{A:STAR}.
Let $d=\dim B_K$.
Let $b$ be a smooth $\Kbar$-point on $B_K$.
If $B$ is a closed subscheme of $\Aff^n$,
a \defi{polydisk neighborhood} of $b$ is a neighborhood $U$ of $b$
in $B(\Kbar)$ in the analytic topology
equipped with, for some $\epsilon>0$, a bijection
\[
	\Delta_{d,\epsilon}:= 
	\{(z_1,\ldots,z_d) \in \Kbar^d : |z_i| \le \epsilon\} 
	\to U
\]
defined by an $n$-tuple of power series in $z_1,\ldots,z_d$
with coefficients in some finite extension of $K$.
(Such neighborhoods exist by the implicit function theorem.
% which can be deduced from~\cite{SerreLieAlgebrasLieGroups}*{Chapter~III,~\S9--\S11}.)
If we replace the embedding $B \injects \Aff^n$, by a different one,
$B \injects \Aff^{n'}$,
the notion of polydisk neighborhood of $b$ changes,
but the new {\em system} of polydisk neighborhoods of $b$
is cofinal with the original one.)
A polydisk neighborhood of $b$ in an arbitrary $B$ is a polydisk
neighborhood of $b$ in some affine open subscheme of $B$.
Let $\scriptH(U)$ be the subring
of $\Kbar[[z_1,\ldots,z_d]]$ consisting of power series $g$
with coefficients in some finite extension of $K$
such that $g$ converges on $\Delta_{d,\epsilon}$.
\end{definition}

The goal of Section~\ref{S:convergent isocrystals} is to prove the following:

\begin{lemma}
\label{L:isocrystal output}
Assume Setup~\ref{A:STAR}.
Let $b_0 \in B(\calO_{\Kbar}) \subseteq B(\Kbar)$
be such that $B_\Kbar$ is smooth at $b_0$.
Then there exists a polydisk neighborhood $U$ of $b_0$
contained in $B(\calO_{\Kbar})$
and a finitely generated $\Z$-submodule $\Lambda \subseteq \scriptH(U)^n$
for some $n$ such that
\[
	\{ b \in U : \rho(\calX_b) > \rho(\calX_{\bar{\eta}}) \}
	= \Union_{\substack{\lambda \in \Lambda \\ \lambda \ne 0}} (\textup{zeros of $\lambda$ in $U$}).
\]
\end{lemma}

Its proof will be completed in Section~\ref{S:proof of lemma}.

\begin{remark}
The analogue over $\C$ is a well-known consequence
of the Lefschetz $(1,1)$ theorem~\cite{Voisin2003vol2}*{\S5.3}.
But the union will often be dense in $B(\C)$,
so this analogue is not useful for our purposes.
\end{remark}

\subsection{Coherent sheaves on $p$-adic formal schemes}

In the next two subsections,
we recall some key notions of \cite{Ogus1984},
specialized slightly to the case we need.
Assume Setup~\ref{A:STAR}.

\begin{definition}[cf.~\cite{Ogus1984}*{\S1}]
A \defi{$p$-adic formal scheme} over $\calO_K$ is
a noetherian formal scheme $T$ of finite type over $\Spf \calO_K$.
(See \cite{EGA-I}*{\S10} for basic definitions regarding formal schemes.)
\end{definition}

\begin{example}
\label{E:p-adic completion of scheme}
Let $X$ be a smooth proper scheme over $\calO_K$.
Its completion with respect to the ideal sheaf $p\OO_X$
is a $p$-adic formal scheme $\widehat{X}$.
``Formal GAGA'' states:
\begin{enumerate}[\upshape (a)]
\item
\label{I:formal GAGA coherent sheaves}
The functor $\Coh(X) \to \Coh(\widehat{X})$
mapping $\FF$ to its $p$-adic completion $\widehat{\FF}$
is an equivalence between
the categories of coherent sheaves~\cite{EGA-III.I}*{Corollaire~5.1.6}.
\item
\label{I:formal GAGA line bundles}
Under this equivalence,
line bundles on $X$ correspond to line bundles on $\widehat{X}$.
\item
\label{I:formal GAGA cohomology}
For any $\FF \in \Coh(X)$ and $q \in \Z_{\ge 0}$,
there is an isomorphism of $\calO_K$-modules
$\HH^q(X,\FF) \isom \HH^q(\widehat{X},\widehat{\FF})$~\cite{EGA-III.I}*{Corollaire~4.1.7}.
\end{enumerate}
\end{example}

We write $K \tensor \cdot$ as an abbreviation for $K \tensor_{\calO_K} \cdot$.
Similarly, $\Kbar \tensor \cdot$ means $\Kbar \tensor_{\calO_K} \cdot$.

\begin{definition}[\cite{Ogus1984}*{Definition~1.1}]
For any $p$-adic formal scheme $T$,
let $\Coh(K \tensor \OO_T)$ denote the full subcategory of
$(K \tensor \OO_T)$-modules
isomorphic to $K \tensor \FF$
for some coherent $\OO_T$-module $\FF$.
Equivalently, we could consider the category 
whose objects are coherent $\OO_T$-modules
but whose set of morphisms from $\FF$ to $\GG$ 
is $\Hom(K \tensor \FF,K \tensor \GG)$.
\end{definition}

\begin{definition}
Similarly, if $B$ is a separated finite-type $\OO_K$-scheme,
we may define $\Coh(K \tensor \OO_B)$
as the category whose objects are coherent $\OO_B$-modules
and whose set of morphisms from $\FF$ to $\GG$ 
is $\Hom(K \tensor \FF,K \tensor \GG)$.
\end{definition}

\begin{proposition}
\label{P:Coh}
Let $B$ be a separated finite-type $\OO_K$-scheme.
\begin{enumerate}[\upshape (a)]
\item\label{I:Coh of generic fiber} 
The functor $\Coh(K \tensor \OO_B) \to \Coh(B_K)$
sending $\FF$ to $\FF|_{B_K}$ is an equivalence of categories.
\item\label{I:completion functor} 
We have a functor 
$\Coh(K \tensor \OO_B) \to \Coh(K \tensor \OO_{\widehat{B}})$
sending $\FF$ to $\hat{\FF}$.
\item\label{I:pullback} 
The resulting functor 
$\Coh(B_K) \to \Coh(K \tensor \OO_{\widehat{B}})$
is compatible with pullback by an $\calO_K$-morphism $B' \to B$
on each side.
\item\label{I:vector spaces} 
If $B=\Spec \OO_K$, then the functor in~\eqref{I:pullback}
is an equivalence of categories.
\end{enumerate}
\end{proposition}

\begin{proof}
This is all straightforward.
In~\eqref{I:vector spaces},
both categories are equivalent
to the category of finite-dimensional $K$-vector spaces.
\end{proof}

\begin{remark}
\label{R:analytic sheaf}
Let $B$ be a separated finite-type $\calO_K$-scheme.
The rigid-analytic generic fiber of the formal scheme $\widehat{B}$
is open in the rigid analytification $(B_K)^\an$ of $B_K$.
So, given a coherent $\OO_B$-module $\FF$
and $b_0 \in B(\calO_K) \subseteq B(K)$ such that $B_K$ is smooth at $b_0$,
we have two routes to construct the ``restriction''
of $\FF$ to a sheaf on a sufficiently small
rigid-analytic neighborhood of $b_0$:
one route goes through $\widehat{\FF}$, and the other goes
through $\FF_K$ on $B_K$.
In particular, if $\FF_K$ is locally free of rank $n$ on $B_K$,
and $U$ is a sufficiently small polydisk neighborhood of $b_0$,
then a choice of local trivialization of $\FF_K$
lets us ``restrict'' global sections of $\Kbar \tensor \widehat{\FF}$
to obtain elements of $\scriptH(U)^n$.
(Although we have used some language of rigid geometry in this remark,
it is not needed anywhere else in this article.)
\end{remark}

\subsection{Definition of convergent isocrystal}

Given a $p$-adic formal $\calO_K$-scheme $T$,
let $T_1$ be the closed subscheme defined by the ideal sheaf $p\OO_T$,
and let $T_0$ be the associated reduced subscheme $(T_1)_\red$.

\begin{definition}[cf.~\cite{Ogus1984}*{Definition~2.1}]
An \defi{enlargement} of $B_\kk$ is a $p$-adic formal $\calO_K$-scheme $T$
equipped with a $\kk$-morphism $z \colon T_0 \to B_\kk$.
A \defi{morphism of enlargements} $(T',z') \to (T,z)$
is an $\calO_K$-morphism $T' \to T$
such that the induced $\kk$-morphism $T'_0 \to T_0$
followed by $z$ equals $z'$.
\end{definition}

\begin{example}
Given $s \in B(\kk)$, let $\lceil s \rceil$
denote the enlargement $(\Spf \calO_K,\Spec \kk \stackrel{s}\to B_\kk)$ of $B_\kk$.
\end{example}

\begin{example}
If $\gamma\colon B' \to B$ is a morphism of $\calO_K$-schemes of finite type,
then we view $\widehat{B}'$ as an enlargement of $B_\kk$
by equipping it with
the $\kk$-morphism $\left(\widehat{B}'\right)_0 = (B'_\kk)_\red \to B_\kk$
induced by $\gamma$.
\end{example}

\begin{definition}[cf.~\cite{Ogus1984}*{Definition~2.7}]
\label{D:convergent isocrystal}
A \defi{convergent isocrystal} on $B_\kk$ consists of the following data:
\begin{enumerate}[\upshape (a)]
\item For every enlargement $(T,z)$ of $B_\kk$,
a sheaf $E_T \in \Coh(K \tensor \OO_T)$.
\item For every morphism of enlargements $g \colon (T',z') \to (T,z)$,
an isomorphism $\theta_g \colon g^* E_T \to E_{T'}$
in $\Coh(K \tensor \OO_T)$.
If $h \colon (T'',z'') \to (T',z')$ is another,
the cocycle condition $\theta_h \circ h^* \theta_g = \theta_{g \circ h}$
is required, and $\theta_{\id} = \id$.
\end{enumerate}
\end{definition}

\subsection{Crystalline cohomology}

\begin{definition}[cf.~\cite{Grothendieck-crystals}*{\S7} and \cite{Berthelot1974}*{III.1.1}]
Let $K$ and $\kk$ be as in Setup~\ref{A:STAR},
and let $W$ be the Witt ring of $\kk$,
so $\calO_K$ is finite as a $W$-module.
Given a smooth proper $\kk$-variety $X$ and $q \in \Z_{\ge 0}$,
we have the \defi{crystalline cohomology} $\HH^q_\cris(X/W)$,
which is a finite $W$-module.
Define
\[
	\HH^q_\cris(X/K):=K \tensor_W \HH^q_\cris(X/W).
\]
There is a \defi{Chern class} homomorphism \cite{Grothendieck-crystals}*{\S7.4}
\[
	c_1^\cris\colon \Pic X \to \HH^2_\cris(X/K).
\]
\end{definition}

\begin{remark}
\label{R:factors through NS}
The fact that crystalline cohomology is
a Weil cohomology \cite{Gillet-Messing1987}
implies \cite{Kleiman1968}*{1.2.1}
that $c_1^\cris(\LL)$ depends only
on the image of $\LL$ in $\NS X_{\overline{\kk}}$.
\end{remark}

If instead of a single $\kk$-variety
we have a family, then the $K$-vector space $\HH^2_\cris(X/K)$ is 
replaced by a compatible system of sheaves,
i.e., a convergent isocrystal:

\begin{theorem}[cf.~\cite{Ogus1984}*{Theorems 3.1 and~3.7}]
\label{T:crystalline isocrystal}
Assume Setup~\ref{A:STAR}.
For each $q \in \Z_{\ge 0}$,
there exists a convergent isocrystal 
$E := \RR^q_\cris f_* \OO_{\calX/K}$ on $B_\kk$
with an isomorphism of $K$-vector spaces
$E_{\lceil s \rceil} \isom \HH^q_\cris(\calX_s/K)$
for each $s \in B(\kk)$.
\end{theorem}

\subsection{De Rham cohomology}

\subsubsection{Algebraic de Rham cohomology}

\begin{definition}
\label{D:algebraic de Rham}
Let $f \colon \calX \to B$ be a smooth proper morphism of noetherian schemes,
and let $q \in \Z_{\ge 0}$.
\defi{De Rham cohomology} 
is defined as the coherent $\OO_B$-module $\R^q f_* \Omega^\bullet_{\calX/B}$ 
 (cf.~\cite{Grothendieck1966-DR}).
It has a \defi{Hodge filtration} in $\Coh(B)$ given by
\[
	\Fil^p( \R^q f_* \Omega^\bullet_{\calX/B} )
	:=
	\im \left( \R^q f_* \Omega^{\ge p}_{\calX/B}
	\to \R^q f_* \Omega^\bullet_{\calX/B} \right).
\]
where $\Omega^{\ge p}$ is obtained from $\Omega^\bullet$
by replacing terms in degrees less than $p$ by $0$.
Also define the coherent $\OO_B$-module
\[
	\HH^{02}(\calX):=
	\frac{\R^2 f_* \Omega^\bullet_{\calX/B}}
	{\Fil^1(\R^2 f_* \Omega^\bullet_{\calX/B})}.
\]
\end{definition}

In the case $B=\Spec K$ for a field $K$,
these $\OO_B$-modules are $K$-vector spaces,
and there is a
\defi{Chern class} homomorphism (cf.~\cite{Hartshorne1975}*{II.7.7})
\[
	c_1^\dR\colon \Pic \calX \to \HH^2_{\dR}(\calX)
		:= \RR^2 f_* \Omega^\bullet_{\calX/K}.
\]
% The map $f \mapsto df/f$ from the complex consisting of $\OO_X^\times$
% in degree $1$ to the complex $\widehat{\Omega}^\bullet_{X/\calO_K}$
% induces a \defi{Chern class} homomorphism

De Rham cohomology over varieties behaves well under pullback 
in characteristic~$0$:
\begin{proposition}
\label{P:pullback of de Rham}
Let $K$ be a field of characteristic~$0$.
Let $B$ be a $K$-variety.
Let $\calX \to B$ be a smooth proper morphism.
\begin{enumerate}[\upshape (a)]
\item\label{I:Deligne locally free}
The $\OO_B$-modules $\R^q f_* \Omega^\bullet_{\calX/B}$
and $\HH^{02}(\calX)$ are locally free.
\item
Let $\alpha \colon B' \to B$ be a morphism of $K$-varieties.
Let $\calX':=\calX \times_B B'$.
Then 
\[
	\R^2 f_* \Omega^\bullet_{\calX'/B'}
	\isom \alpha^* \R^2 f_* \Omega^\bullet_{\calX/B},
\]
and this isomorphism respects the Hodge filtrations,
so we also have
\[
	\HH^{02}(\calX') \isom \alpha^* \HH^{02}(\calX).
\]
\end{enumerate}
\end{proposition}

\begin{proof}
See~\cite{Deligne1968}*{Th\'eor\`eme~5.5}.  
\end{proof}

\subsubsection{Formal de Rham cohomology}

There is an analogous definition of de Rham cohomology for formal schemes:
\begin{definition}
Let $g \colon \calY \to T$ be a smooth proper morphism
of $p$-adic formal schemes
and let $q \in \Z_{\ge 0}$.
The sheaf
$K \tensor \R^q g_* \widehat{\Omega}^\bullet_{\calY/T}
\in \Coh(K \tensor \OO_T)$
has a Hodge filtration in $\Coh(K \tensor \OO_T)$ given by
\[
	\Fil^p( K \tensor \R^q g_* \widehat{\Omega}^\bullet_{\calY/T} )
	:=
	\im \left( K \tensor \R^q g_* \widehat{\Omega}^{\ge p}_{\calY/T}
	\to K \tensor \R^q g_* \widehat{\Omega}^\bullet_{\calY/T} \right).
\]
Also define the sheaf
\[
	\HH^{02}(\calY/K):=
	\frac{K \tensor \R^2 g_* \widehat{\Omega}^\bullet_{\calY/T}}
	{\Fil^1(K \tensor \R^2 g_* \widehat{\Omega}^\bullet_{\calY/T})}.	
\]
\end{definition}

In the case $T=\Spf \calO_K$, we have
\[
	c_1^\dR\colon \Pic \calY \to \HH^2_\dR(\calY/K):=
	K \tensor \R^2 g_* \widehat{\Omega}^\bullet_{\calY/T}.
\]

\subsubsection{Fibers of de Rham cohomology}

For families arising
as the $p$-adic completion of a
smooth proper morphism of $\calO_K$-schemes,
we show that taking de Rham cohomology 
commutes with restriction to fibers:
\begin{proposition}
\label{P:arising from algebra}
Let $f \colon \calX \to B$ be a smooth proper morphism
of $\calO_K$-schemes.
Let $b \in B(\calO_K)$.
and let $\calX_b$ be the pullback of $\calX$ 
by $\Spec \calO_K \stackrel{b}\to B$.
and $\calX_{b,K} = \calX_b \times K$.
Then 
\begin{enumerate}[\upshape (a)]
\item \label{I:fibers of de Rham cohomology} 
For each $q \ge 0$,
there are isomorphisms of filtered $K$-vector spaces
\[
	K \tensor \left. 
	\R^q f_* \widehat{\Omega}^\bullet_{\widehat{\calX}/\widehat{B}} \right|_b
	\to
	\HH^q_{\dR}(\widehat{\calX}/K).
\]
\item \label{I:isomorphism of H02}
There is an isomorphism of $K$-vector spaces
\[
	\left.\HH^{02}(\widehat{\calX}/K)\right|_b
	\to \HH^{02}(\widehat{\calX_b}/K).
\]
\end{enumerate}
\end{proposition}

\begin{proof}
\hfill
\begin{enumerate}[\upshape (a)]
\item 
The algebraic analogue of this isomorphism, namely
\[
	\left. \left( \R^q f_* \Omega^\bullet_{\calX_K/B_K} \right) \right|_{b_K}
	\to \HH^q_{\dR}(\calX_K),
\]
is an isomorphism of filtered $K$-vector spaces 
by Proposition~\ref{P:pullback of de Rham} 
with $\alpha = (b)_K \colon \Spec K \to B_K$.
Now apply the functor in Proposition~\ref{P:Coh}\eqref{I:vector spaces}.
\item
This follows from~\eqref{I:fibers of de Rham cohomology} for $q=2$.
\end{enumerate}
\end{proof}

\subsection{Comparison and the $p$-adic Lefschetz $(1,1)$ theorem}

The following result identifies crystalline and de Rham cohomologies,
even in the family setting.

\begin{theorem}
\label{T:dR cris comparison}
Assume Setup~\ref{A:STAR}.
Let $(T,z)$ be an enlargement of $B_\kk$.
Let $f_0\colon \calX_0 \to T_0$ be obtained from $f \colon \calX \to B$
by base change along $z \colon T_0 \to B_\kk \injects B$.
Let $g \colon \calY \to T$ be a smooth proper lifting of $f_0$.
Then for each $q \in \Z_{\ge 0}$ there is a canonical isomorphism
\begin{equation}
\label{E:crystalline de Rham comparison}
	\sigma_{\cris,T} \colon
	K \tensor \R^q g_* \widehat{\Omega}^\bullet_{\calY/T}
	\to \left( \RR^q_\cris f_* \OO_{\calX/K} \right)_T.
\end{equation}
Moreover, if $t \in T(\calO_K)$
and $s = z(t(\Spec \kk)) \in B(\kk)$,
then the isomorphism $\sigma_{\cris,t}$ induced by $\sigma_{\cris,T}$
on the fibers above $t$
fits in a commutative diagram
\begin{equation}
\label{E:commutative square}
\xymatrix{
	\Pic \calY_t \ar[d]_{c_1^\dR} \ar[r] & \Pic \calX_s \ar[d]^{c_1^\cris} \\
       \HH^2_\dR(\calY_t/K) \ar[r]^{\sigma_{\cris,t}} & \HH^2_\cris(\calX_s/K).
}
\end{equation}
\end{theorem}

\begin{proof}
For \eqref{E:crystalline de Rham comparison}, 
see~\cite{Ogus1984}*{Theorem~3.8.2}.
For \eqref{E:commutative square}, see
\cite{Berthelot-Illusie1970}*{2.3} and~\cite{Berthelot-Ogus1983}*{3.4}.
\end{proof}

Finally, we have what one might call a $p$-adic analogue
of the Lefschetz $(1,1)$ theorem:

\begin{theorem}[cf.~\cite{Berthelot-Ogus1983}*{Theorem~3.8}]
\label{T:p-adic Lefschetz 1,1 theorem}
Let $X \stackrel{g}\to \Spf \calO_K$
be a smooth proper $p$-adic formal scheme.
Let $\LL_\kk$ be a line bundle on $X_\kk$.
Then the following are equivalent:
\begin{enumerate}[\upshape (a)]
\item There exists $m$ such that $\LL_\kk^{\tensor p^m}$
lifts to a line bundle on $X$.
\item
The element of $\HH^2_\dR(X/K):=K \tensor \R^q g_* \widehat{\Omega}^\bullet_{X/\calO_K}$
corresponding under $\sigma_{\cris,\calO_K}$ to $c_1(\LL_\kk) \in \HH^2_\cris(X_\kk/K)$
maps to $0$ in the quotient $\HH^{02}(X/K)$.
\end{enumerate}
\end{theorem}

All the above definitions and results are compatible with
respect to base change from $\calO_K$ to $\calO_L$
for a finite extension $L$ of $K$ \cite{Ogus1984}*{3.6, 3.9.1, 3.11.1}.

\subsection{Proof of Lemma~\ref{L:isocrystal output}}
\label{S:proof of lemma}

By smoothness, there is a unique irreducible component of $B_\Kbar$ containing
the point of $B(\Kbar)$ corresponding to $b_0$.
Replace $K$ by a finite extension
so that $b_0$ and this component are defined over $K$,
and replace $B$ by the closure of this component.
Then replace $B$ by an open subscheme
to assume that $B$ is a closed subscheme of
$\Aff^r = \Spec \calO_K[x_1,\ldots,x_r]$
for some $r$.
Translate so that $b_0$ is the origin in $\Aff^r$.
Let $s \in B(\kk)$ be the reduction of $b_0$, the origin in $\Aff^r(\kk)$.

Let $\wp \colon \Aff^r \to \Aff^r$
be the morphism induced by the $\calO_K$-algebra homomorphism
mapping each variable $x_i$ to $p x_i$.
Let $B' = \wp^{-1}(B)$.
Let $b_0' \in B'(\calO_K)$ be the origin, so $\wp(b_0')=b_0$.
Let $T = \widehat{B'}$.
Pulling back $f \colon \calX \to B$
yields a morphism of $\calO_K$-schemes $\calX' \to B'$.
Completing yields
a morphism of $p$-adic formal $\calO_K$-schemes $\calX_T \to T$.
We write $f$ for any of these.

Let $E$ be the convergent isocrystal $\RR^2_\cris f_* \OO_{\calX/K}$ on $B_\kk$
given by Theorem~\ref{T:crystalline isocrystal}.
Because the special fiber of $T$ maps to $s \in B(\kk)$,
we have a morphism of enlargements $T \to \lceil s \rceil$,
so the definition of convergent isocrystal gives an identification
\[
	E_T \isom E_{\lceil s \rceil} \tensor_{\calO_K} \OO_T
	\isom \HH^2_\cris(\calX_s/K) \tensor_{\calO_K} \OO_T,
\]
and the latter is a trivialized sheaf in $\Coh(K \tensor \OO_T)$.

Let $\LL_\kk$ be a line bundle on $\calX_s$.
Then $c_1^\cris(\LL_\kk) \in \HH^2_\cris(X_s/K)$
gives rise to a constant section 
$\gamma_\cris(\LL_\kk) := c_1^\cris(\LL_\kk) \tensor 1
\in \HH^2_\cris(\calX_s/K) \tensor_{\calO_K} \OO_T \isom E_T$.
Applying $\sigma_{\cris,T}^{-1}$ yields
a section $\gamma_\dR(\LL_\kk)$ of
$K \tensor \R^2 f_* \widehat{\Omega}^\bullet_{\calX_T/T}$,
which can be mapped to a section $\gamma_{02}(\LL_\kk)$
of the quotient sheaf $\HH^{02}(\calX_T/K)$.

Let $b' \in B'(\calO_K)$.
Let $\calX_{b'}$ be the $\calO_K$-scheme
obtained by pulling back $\calX \to B$
by the composition $\Spec \calO_K \stackrel{b'}\to B' \to B$.
Let $\calX_{b',K} = \calX_{b'} \times K$.
We can ``evaluate''
$\gamma_\cris(\LL_\kk)$, $\gamma_\dR(\LL_\kk)$, and $\gamma_{02}(\LL_\kk)$ at $b'$
by pulling back via $\Spf \calO_K \stackrel{b'}\to T$
to obtain values in $K$-vector spaces
\begin{align*}
	\gamma_\cris(\LL_\kk,b') &\in \HH^2_\cris(\calX_s/K), \\
	\gamma_\dR(\LL_\kk,b') &\in \HH^2_\dR(\widehat{\calX_{b'}}/K),\text{ and}\\
	\gamma_{02}(\LL_\kk,b') &\in \HH^{02}(\widehat{\calX_{b'}}/K).
\end{align*}
Because the composition of enlargements
$\widehat{b'} \to T \to \lceil s \rceil$
is the identity, the cocycle condition
in Definition~\ref{D:convergent isocrystal}
yields $\gamma_\cris(\LL_\kk,b') = c_1^\cris(\LL_\kk)$.

Everything so far has been compatible with base extension from $\calO_K$
to $\calO_L$ for a finite extension $L$ of $K$,
and we may take direct limits to adapt the definitions and results
above to $\calO_{\Kbar}$.

Proposition~\ref{P:construction of specialization 2}\eqref{I:p>0}
gives an injective homomorphism
\[
	\spe_{\bar{b}',\bar{s}} \colon
	(\NS \calX_{b',\Kbar})_\Q
	\injects
	(\NS \calX_{\bar{s}})_\Q.
\]

\begin{claim}
\label{C:lifting criterion 1}
The class $[\LL_\kk] \in (\NS \calX_{\bar{s}})_\Q$
is in the image of $\spe_{\bar{b}',\bar{s}}$
if and only if $\gamma_{02}(\LL_\kk,b')=0$.
\end{claim}

\begin{proof}
Suppose that $[\LL_\kk]$ is in the image of $\spe_{\bar{b}',\bar{s}}$.
After replacing $\LL_\kk$ by a tensor power,
replacing $K$ by a finite extension,
and tensoring $\LL_\kk$ with a line bundle algebraically equivalent to $0$
(which, by Remark~\ref{R:factors through NS}, does not change any of
the sections and values $\gamma_{\cdot}(\cdot)$),
we may assume that the isomorphism class of $\LL_\kk$ in $\Pic \calX_s$
is the specialization of the isomorphism class of some
line bundle $\LL_K$ on $\calX_{b',K}$.
Lift $\LL_K$ to a line bundle $\LL$ on the $\calO_K$-scheme $\calX_{b'}$.
Completing yields $\widehat{\LL} \in \Pic \widehat{\calX_{b'}}$.
The commutative diagram in Theorem~\ref{T:dR cris comparison}
shows that the element
$c_1^\dR(\widehat{\LL}) \in \HH^2_\dR(\widehat{\calX_{b'}}/K)$
is mapped by $\sigma_{\cris,b'}$ to $c_1^\cris(\LL_\kk)$.
Then Theorem~\ref{T:p-adic Lefschetz 1,1 theorem}
applied to $\widehat{\calX_{b'}}$
shows that $\gamma_{02}(\LL_\kk,b')=0$.

Conversely, suppose that $\gamma_{02}(\LL_\kk,b')=0$.
Theorem~\ref{T:p-adic Lefschetz 1,1 theorem}
applied to $\widehat{\calX_{b'}}$
shows that after raising $\LL_\kk$ to a power of $p$,
we have that $\LL_\kk$ comes from some $\widehat{\LL}$ on $\widehat{\calX_{b'}}$.
By
Example~\ref{E:p-adic completion of scheme}\eqref{I:formal GAGA line bundles},
$\widehat{\LL}$ comes from some $\LL$ on $\calX_{b'}$.
Then $[\LL_\kk] = \spe_{\bar{b}',\bar{s}}([\Kbar \tensor \LL])$.
This completes the proof of Claim~\ref{C:lifting criterion 1}.
\end{proof}

Because of Remark~\ref{R:factors through NS},
$\gamma_{02}$ on $\Pic \calX_{\bar{s}}$ induces a homomorphism
from $\NS \calX_{\bar{s}}$ to the space of sections
of the sheaf $\Kbar \tensor_K \HH^{02}(\calX_T/K)$ on $T$.
Let $\Lambda_T$ be the image,
so $\Lambda_T$ is a finitely generated $\Z$-module.
For any $b' \in B'(\calO_\Kbar)$,
evaluation at $b'$ as in 
Proposition~\ref{P:arising from algebra}\eqref{I:isomorphism of H02}
defines a homomorphism $\ev_{b'}$
from $\Lambda_T$ (or $(\Lambda_T)_\Q$)
to $\HH^{02}(\widehat{\calX_{b'}}/\Kbar)$.

Applying Claim~\ref{C:lifting criterion 1}
over all finite extensions of $\calO_K$ yields

\begin{corollary}
\label{C:formula}
\hfill
\begin{enumerate}[\upshape (a)]
\item
\label{I:rho formula}
For any $b' \in B'(\calO_{\Kbar})$, $\rho(\calX_{b'})$
is the rank of the kernel of the composition
\[
\xymatrix{
	(\NS \calX_{\bar{s}})_\Q \ar@{->>}[r]^-{\gamma_{02}} \ar@/_1pc/[rr]_-{\gamma_{02}(-,b')} & (\Lambda_T)_\Q \ar[r]^-{\ev_{b'}} & \HH^{02}(\widehat{\calX_{b'}}/\Kbar). \\
}
\]
\item
\label{I:condition for rho equality}
In particular,
\begin{equation}
\label{E:rho lower bound}
	\rho(\calX_{b'}) \ge \rk \ker \gamma_{02},
\end{equation}
with equality if and only if $\ev_{b'} \colon \Lambda_T \to \HH^{02}(\widehat{\calX_{b'}}/\Kbar)$
is injective.
\end{enumerate}
\end{corollary}

Proposition~\ref{P:pullback of de Rham}\eqref{I:Deligne locally free} 
lets us apply Remark~\ref{R:analytic sheaf}
to $\FF:=\RR^2 f_* \OO_{\calX'}$ on $B'$
to obtain a polydisk neighborhood $U'$ of $b_0'$ in $B'(\calO_\Kbar)$
such that the subgroup $\Lambda_T$ of global sections
of $\Kbar \tensor \HH^{02}(\calX_T/K)$
is expressed on $U'$ as a subgroup $\Lambda'$ of $\scriptH(U')^n$:
in fact, if $K$ is enlarged so that all elements of $\NS(\calX_{\bar{s}})$
are defined over the residue field $\kk$,
then the coefficients of the elements in $\Lambda'$ lie in $K$.
For $b' \in U'$, we may interpret $\ev_{b'}$ concretely in terms of
evaluation of power series in $\Lambda'$.

We will prove
\begin{equation}
\label{E:generic rank}
	\rk \ker \gamma_{02} = \rho(\calX_{\bar{\eta}})
\end{equation}
by comparing both with $\rho(\calX_\beta)$ for a ``very general'' $\beta \in B'(\OO_\Kbar)$.

\begin{lemma}
\label{L:countable union}
Let $Z$ be a finite-type $K$-scheme that is smooth of dimension $n$.
Fix $z_0 \in Z(K)$.
Then no countable union of subschemes of $Z$ of dimension less than $n$
can contain a neighborhood of $z_0$ in $Z(K)$.
\end{lemma}

\begin{proof}
Shrink $Z$ so that there is an \'{e}tale morphism $\pi\colon Z \to \Aff^n$.
By the definition of \'etale morphism and 
the nonarchimedean implicit function theorem~\cite{Igusa2000}*{Theorem~2.1.1},
$\pi$ maps any sufficiently small neighborhood of $z_0$ in $Z(K)$
bijectively to a neighborhood of $\pi(z_0)$ in $\Aff^n(K)$.
Also, the scheme-theoretic image in $\Aff^n$ 
of any subscheme in $Z$ of dimension less than $n$
is of dimension less than $n$.
So we may reduce to the case $Z=\Aff^n$.
This follows by induction on $n$
by projecting onto $\Aff^{n-1}$ and using the uncountability of $K$.
\end{proof}

Corollary~\ref{C:jumping locus is countable union}
and Lemma~\ref{L:countable union}
show that there exists $\beta \in B'(\OO_\Kbar)$ near $b_0'$
such that $\rho(\calX_\beta) = \rho(\calX_{\bar{\eta}})$.
For any $b' \in B'(\OO_\Kbar)$, we have
\begin{equation}
\label{E:key inequalities}
	\rk \ker \gamma_{02} \le \rho(\calX_\beta) = \rho(\calX_{\bar{\eta}}) \le \rho(\calX_{b'})
\end{equation}
by \eqref{E:rho lower bound}, the choice of $\beta$, and
Proposition~\ref{P:construction of specialization 2}\eqref{I:rho inequality},
respectively.

\begin{claim}
\label{C:one point}
For some $b' \in B'(\OO_\Kbar)$, we have $\rk \ker \gamma_{02} = \rho(\calX_{b'})$.
\end{claim}

\begin{proof}
Applying Proposition~\ref{P:evaluation is injective}\footnote{The proofs in
Section~\ref{S:unions of zero loci} do not rely on any results in
this section.}
to $\Lambda'$
gives a nonempty open subset $V$ of the polydisk in $C^d$
such that $\ev_u$ is injective for $u \in V$.
Since $\Kbar$ is dense in $C$, the set $V$ contains a point in $\Kbar^d$,
and we let $b'$ be its image in $B'(\calO_\Kbar)$,
so $\ev_{b'}$ is injective.
Apply Corollary~\ref{C:formula}\eqref{I:condition for rho equality} to $b'$.
\end{proof}

Applying \eqref{E:key inequalities} with $b'$ as in Claim~\ref{C:one point}
proves \eqref{E:generic rank}.

Next, $\wp$ maps $U'$ isomorphically to
a polydisk neighborhood $U$ of $b_0$ in $B(\calO_{\Kbar})$,
and $\Lambda'$ corresponds to some $\Lambda \subseteq \scriptH(U)^n$.
Substituting \eqref{E:generic rank} into
Corollary~\ref{C:formula}\eqref{I:condition for rho equality},
expressed on $U$ in terms of $\Lambda$, shows that for $b \in U$,
\[
	\rho(\calX_b) \ge \rho(\calX_{\bar{\eta}}),
\]
with equality if and only if
$\lambda(b) \ne 0$ for every nonzero $\lambda \in \Lambda$.
This completes the proof of Lemma~\ref{L:isocrystal output}.

\begin{remark}
We conjecture that Lemma~\ref{L:isocrystal output}
holds with $B(\calO_\Kbar)$ replaced by $B(\Kbar)$,
but crystalline methods do not suffice to prove this.
This would imply that Theorem~\ref{T:local} holds with $B(C)$
in place of $B(\calO_C)$.
\end{remark}

%****************************************************************************
\section{Unions of zero loci of power series}
\label{S:unions of zero loci}

To understand the nature of the following statement,
the reader is urged to consider the $r=1$ case first.

\begin{proposition}
\label{P:evaluation is injective}
Let $C$ be as in Setup~\ref{A:STAR}.
Let $D = \{(z_1,\ldots,z_d) \in C^d : v(z_i) \le \epsilon \textup{ for all $i$}\}$
for some $\epsilon>0$.
Let $R$ be the subring of $C[[z_1,\ldots,z_d]]$
consisting of power series that converge on $D$.
Let $r$ be a nonnegative integer.
Let $\Lambda$ be a finite-dimensional $\Q_p$-subspace of $R^r$.
Then there exists a nonempty analytic open subset $U$ of $D$
such that for all $u \in U$, the evaluation-at-$u$ map
\begin{align*}
	\ev_u \colon \Lambda &\to C^r \\
	(f_1,\ldots,f_r)
		& \mapsto (f_1(u),\ldots,f_r(u))
\end{align*}
is injective.
\end{proposition}

\begin{remark}
The archimedean analogue of Proposition~\ref{P:evaluation is injective}
is false.
For example, if $\Lambda$ is the $\R$-span of $1,x,x^2$,
then there is no $u \in \C$ such that
the evaluation-at-$u$ map $\Lambda \to \C$ is injective.
Even if we consider only the $\Z$-span of $1,x,x^2$,
the evaluation-at-$u$ map fails to be injective on a dense subset of $\C$.
\end{remark}

The rest of this section is devoted to the proof of
Proposition~\ref{P:evaluation is injective}.
The proof is by induction on $r$.
Because the base case $r=1$ is rather involved,
we begin by explaining the inductive step.

Suppose that $r>1$, and that Proposition~\ref{P:evaluation is injective}
is known for $r'<r$.
Let $\pi\colon R^r \to R^{r-1}$ be the projection to
the first $r-1$ coordinates.
Let $\Lambda^{(r)}$ and $\Lambda_{r-1}$ be the kernel and image of $\pi|_\Lambda$.
View $\Lambda^{(r)}$ as a subgroup of $R$.
For any $u \in D$, we have a commutative diagram with exact rows
\begin{equation}
\label{E:Lambdas}
\xymatrix{
0 \ar[r] & \Lambda^{(r)} \ar[r] \ar[d]^{\ev_u} & \Lambda \ar[r] \ar[d]^{\ev_u} & \Lambda_{r-1} \ar[r] \ar[d]^{\ev_u} & 0 \\
0 \ar[r] & C \ar[r] & C^r \ar[r] & C^{r-1} \ar[r] & 0. \\
}
\end{equation}
The inductive hypothesis applied to $\Lambda_{r-1} \subseteq R^{r-1}$
gives a closed polydisk $D' \subseteq D$
such that the right vertical map in~\eqref{E:Lambdas} is injective
for all $u \in D'$.
The inductive hypothesis applied to $\Lambda^{(r)} \subseteq R$
gives a closed polydisk $D'' \subseteq D'$
such that the left vertical map in~\eqref{E:Lambdas} is injective
for all $u \in D''$.
Then for $u \in D''$, the middle vertical map in~\eqref{E:Lambdas}
is injective.
This completes the proof of the inductive step.

Before discussing the base case $r=1$,
we prove another lemma.
Let $v \colon C \to \Q \union \{+\infty\}$ be the valuation on $C$,
normalized by $v(p)=1$.
If $\vec{t} = (t_1,\ldots,t_s) \in C^s$ for some $s$,
define $v(\vec{t}):=\min_j v(t_j)$.

\begin{lemma}
\label{L:Qp linear combinations}
\hfill
  \begin{enumerate}[\upshape (a)]
  \item
\label{I:independent vectors}
If $\vec{t}_1,\ldots,\vec{t}_n \in C^s$ are $\Q_p$-independent,
then $\left\{v\left(\sum a_i \vec{t}_i\right) : (a_1,\ldots,a_n) \in (\Z_p)^n - (p\Z_p)^n\right\}$
is finite.
  \item
\label{I:image in Q/Z}
If $t_1,\ldots,t_n \in C$, then $\left\{v(t): t = \sum a_i t_i \ne 0 \textup{ for some $a_i \in \Q_p$}\right\}$ has finite image in $\Q/\Z$.
  \end{enumerate}
\end{lemma}

\begin{proof}\hfill
  \begin{enumerate}[\upshape (a)]
  \item The function
    \begin{align*}
      (\Z_p)^n - (p\Z_p)^n &\to \Q \\
	(a_1,\ldots,a_n) &\mapsto v\left(\sum a_i \vec{t}_i\right)
    \end{align*}
is continuous for the $p$-adic topology on the left
and the discrete topology on the right, so by compactness its image is finite.
\item By replacing the $t_i$ with a basis for their $\Q_p$-span, we reduce to the $s=1$ case of~\eqref{I:independent vectors}.\qedhere
  \end{enumerate}
\end{proof}

{}From now on, we assume $r=1$; i.e., $\Lambda \subseteq R$.
Choose a $\Q_p$-basis $f_1,\ldots,f_n$ of $\Lambda$.
We may assume that $D$ is the unit polydisk,
so the coefficients of each $f_i$ tend to $0$.
Multiply all the $f_i$ by a single power of $p$
to assume that $f_i \in \calO_C[[z_1,\ldots,z_d]]$.
For some $m$, the images of $f_i$ in $C[[z_1,\ldots,z_d]]/(z_1,\ldots,z_d)^m$
are $\Q_p$-independent,
because a descending sequence of vector spaces in $\Lambda$
with zero intersection must be eventually zero.
Fix such an $m$.

Let $M$ be the set of monomials $\mu$ in the $z_i$
whose total degree $\deg \mu$ is less than $m$.
For each $\mu$, let $c_i^\mu \in C$ be the coefficient of $\mu$ in $f_i$.
For each $\mu \in M$,
apply Lemma~\ref{L:Qp linear combinations}\eqref{I:image in Q/Z}
to $c_1^\mu,\ldots,c_n^\mu$
to obtain a finite subset $S_\mu$ of $\Q/\Z$.
Let $S = \Union_{\mu \in M} S_\mu$.
Let $q_1,\ldots,q_d$ be distinct primes greater than $m$
that do not appear in the denominators
of rational numbers representing elements of $S$.
For $i=1,\ldots,n$, let $\vec{t}_i \in C^{\#M}$
be the vector whose coordinates are the $c_i^\mu$ for $\mu \in M$.
By choice of $m$, the $\vec{t}_i$ are $\Q_p$-independent.
By Lemma~\ref{L:Qp linear combinations}\eqref{I:independent vectors},
the set
\[	
	\left\{v\left(\sum a_i \vec{t}_i\right) :
	(a_1,\ldots,a_n) \in (\Z_p)^n - (p\Z_p)^n\right\}
\]
is finite; choose $A \in \Q$ larger than all its elements.
Choose a positive integer $N$ such that
\begin{equation}
\label{E:definition of N}
	mN > (m-1)(N+1/q_i) + A
\end{equation}
for all $i$.
Let
\[
	U:=\{(z_1,\ldots,z_n) \in C^n
	: v(z_i) = N + 1/q_i \text{ for all $i$}\},
\]
so $U$ is open in $D$.

Consider an arbitrary nonzero element
$f = \sum_{\text{all $\mu$}} c^\mu \mu$ of $\Lambda$.
So $f = \sum_{i=1}^n a_i f_i$ for some
$(a_1,\ldots,a_n) \in \Q_p^n - \{\vec{0}\}$.
Let $u = (u_1,\ldots,u_d) \in U$.
It remains to show that $f(u) \ne 0$.

By multiplying $f$ by a power of $p$, we may assume that
$(a_1,\ldots,a_n) \in (\Z_p)^n - (p\Z_p)^n$.
If $\mu = z_1^{e_1}\cdots z_d^{e_d}$ and $\deg \mu \ge m$,
then
\begin{equation}
\label{E:valuation1}
	v(c^\mu \mu(u)) \ge 0 + e_1 N + \cdots + e_d N \ge mN
\end{equation}
On the other hand, the definition of $A$ yields 
$\xi = z_1^{e_1}\cdots z_d^{e_d} \in M$
such that $v(c^{\xi}) < A$, so that
\begin{equation}
\label{E:valuation2}
	v(c^{\xi} \xi(u)) < A + \sum_{i=1}^d e_i(N+1/q_i) \le mN
\end{equation}
by~\eqref{E:definition of N}.

To show that $f(u) \ne 0$, it remains to show that
the valuations $v(c^\mu \mu(u))$ for $\mu \in M$ such that $c^\mu \ne 0$
are distinct,
since then the minimum of these is finite and equals $v(f(u))$,
by \eqref{E:valuation1} and~\eqref{E:valuation2}.
In fact, if $\mu = z_1^{e_1}\cdots z_d^{e_d}$ and $\deg \mu < m$ and $c^\mu \ne 0$,
then for each $i$ the choice of the $q_i$
implies $v(c^\mu) \in \Z_{(q_i)}$
(i.e., $q_i$ does not divide the denominator of $v(c^\mu)$),
so
\[
	v(c^\mu \mu(u)) \in \frac{e_i}{q_i} + \Z_{(q_i)};
\]
moreover $e_i \le \deg \mu < m < q_i$,
so $e_i$ is determined by $v(c^\mu \mu(u))$ whenever $c^\mu \ne 0$.
This completes the proof of Proposition~\ref{P:evaluation is injective}.

%****************************************************************************
\section{Proof of Theorems \ref{T:local} and~\ref{T:global}}
\label{S:proof of local and global theorems}

\begin{lemma}
\label{L:density of points}
Let $k \subseteq k'$ be an extension of algebraically closed valued fields
such that $k$ is dense in $k'$.
Then for any finite-type $k$-scheme $B$, the set $B(k)$ is dense in $B(k')$
with respect to the analytic topology.
\end{lemma}

\begin{proof}
Let $b' \in B(k')$ be a point to be approximated by $k$-points.
We may replace $B$ by the Zariski closure of the image of $b'$
under $B_{k'} \to B$.
Then $b'$ is a smooth point.
We may shrink $B$ to assume that $B$ is 
finite \'{e}tale over an open subscheme of $\Aff^n$ for some $n$.
Since $k$ is algebraically closed, we may reduce to the case
that $B$ is open in $\Aff^n$.
This case follows from $k$ being dense in $k'$.
\end{proof}

\begin{proof}[Proof of Theorem~\ref{T:local}]
Let $b_0 \in B(\calO_C)$.
We need to show that any neighborhood $U_0$ of $b_0$ in $B(\calO_C)$
contains a nonempty open set $V$ that does not meet $B(\calO_C)_{\jumping}$.
By Lemma~\ref{L:density of points} we may assume that $b_0 \in B(\calO_\Kbar)$.
Lemma~\ref{L:isocrystal output} gives a smaller neighborhood
$U_1$ of $b_0$ in $B(\calO_C)$ in which the jumping locus
is described explicitly in terms of power series,
and Proposition~\ref{P:evaluation is injective}
gives a nonempty open subset $V$ of $U_1$
such that $\rho(\calX_b)=\rho(\calX_{\bar{\eta}})$
for all $b \in V \intersect B(\calO_\Kbar)$.

Suppose that $b \in V \intersect B(\calO_C)_{\jumping}$.
By Corollary~\ref{C:jumping locus is countable union},
$b$ is contained in a $\Kbar$-variety $Z$
such that $Z(C) \subseteq B(C)_{\jumping}$.
Lemma~\ref{L:density of points} gives
$b' \in V \intersect Z(\Kbar) \subseteq B(\OO_\Kbar)_{\jumping}$,
which contradicts the definition of $V$.
\end{proof}

\begin{proof}[Proof that Theorem~\ref{T:local} implies Theorem~\ref{T:global}]
Assume that $k$ and $\calX \to B$ are as in Theorem~\ref{T:global}.
Replacing $B$ by a dense open subscheme,
we may assume that $B$ is smooth over $k$.
Choose a finitely generated $\Z$-algebra $A$ in $k$
such that $\calX \to B$ is the base extension of a morphism
$\calX_A \to B_A$ of $A$-schemes.
By localizing $A$, we may assume that $\calX_A \to B_A$
is a smooth proper morphism, and that $B_A \to \Spec A$
is smooth with geometrically irreducible fibers 
\cite{EGA-IV.III}*{8.10.5(xii),~9.7.7(i)}, \cite{EGA-IV.IV}*{17.7.8(ii)}.
Because of Proposition~\ref{P:algebraically closed base change},
we may replace $k$ by the algebraic closure of $\Frac(A)$ in $k$.

By \cite{CasselsLocalFields1986}*{Chapter~5,~Theorem~1.1},
there exists an embedding $A \injects \Z_p$ for some prime $p$.
Extend it to an embedding $k \injects \C_p$.
Base extend by $A \injects \Z_p$ to obtain $\calX_{\Z_p} \to B_{\Z_p}$.
Since $B_{\Z_p} \to \Spec \Z_p$
is smooth with geometrically irreducible special fiber,
the set $B_{\Z_p}(\calO_{\C_p})$ is nonempty.
Apply Theorem~\ref{T:local} to $\calX_{\Z_p} \to B_{\Z_p}$
to find a nonempty open subset $U$ of
$B_{\Z_p}(\calO_{\C_p}) \subset B_{\Z_p}(\C_p) = B(\C_p)$
in the analytic topology such that $\rho(\calX_b) = \rho(\calX_{\bar{\eta}})$
for all $b \in U$.
The field $k$ is dense in $\C_p$ since even $\Qbar$ is dense in $\C_p$,
so Lemma~\ref{L:density of points} shows that $U$ contains a point $b$ of $B(k)$;
this $b$ is as required in Theorem~\ref{T:global},
because of Proposition~\ref{P:algebraically closed base change}.
\end{proof}

%****************************************************************************
\section{Projective vs.\ proper}
\label{S:projective vs proper}

\begin{theorem}
\label{T:projective vs proper}
Let $B$ be a variety over a field $k$ of characteristic~$0$.
Let $\calX \to B$ be a smooth proper morphism
such that all closed fibers are projective.
Then there exists a Zariski dense open subvariety $U$ of $B$
such that $\calX_U \to U$ is projective.
\end{theorem}

\begin{proof}
We may assume that $B$ is irreducible, say with generic point $\eta$.
Theorem~\ref{T:global} yields a closed point $b \in B$
such that a specialization map $\NS \calX_{\bar{\eta}} \to \NS \calX_{\bar{b}}$
is an isomorphism.
By assumption, $\calX_b$ is projective, so we may choose
an ample line bundle $\LL_{\bar{b}}$ on $\calX_{\bar{b}}$.
By construction of $b$,
the class $[\LL_{\bar{b}}] \in \NS \calX_{\bar{b}}$ comes from
the class in $\NS \calX_{\bar{\eta}}$
of some line bundle $\LL$ on $\calX_{\bar{\eta}}$.
By Proposition~\ref{P:construction of specialization 2}\eqref{I:specializing ampleness}, $\LL$ is ample.
Then $\LL$ comes from a line bundle on $\calX_L$ for some finite extension
$L$ of $\kappa(B)$, so $\calX_L$ is projective,
so $\calX_\eta$ is projective \cite{EGA-II}*{6.6.5}.
Finally, spreading out shows that $\calX_U \to U$ is projective
for some dense open subscheme $U$ of $B$ \cite{EGA-IV.III}*{8.10.5(xiii)}.
\end{proof}

\begin{remark}
Under the hypotheses of Theorem~\ref{T:projective vs proper},
we may also deduce that {\em every} fiber of $\calX \to B$ is projective:
just apply Theorem~\ref{T:projective vs proper} to each irreducible
subvariety of $B$.
But we {\em cannot} deduce that $\calX \to B$ is projective,
or even that $B$ can be covered by open sets $U$ such that $\calX_U \to U$
is projective, as the following example of Raynaud shows.
\end{remark}

\begin{example}[\cite{Raynaud1970faisceaux}*{XII.2}]
\label{E:Raynaud}
Let $A$ be a nonzero abelian variety over a field $k$.
Let $\sigma$ be the automorphism $(x,y) \mapsto (x,x+y)$ of $A \times A$.
Let $B$ be a rational curve whose singular locus is a single node,
and let $\tilde{B}$ be its normalization.
Then one can construct an abelian scheme $\calA \to B$
whose base extension to $\tilde{B}$ is simply $(A \times A) \times \tilde{B}$
with the fibers above the two preimages of the node identified via $\sigma$.
Moreover, for any neighborhood $U$ of the node, $\calA_U \to U$
is not projective.
\end{example}

\begin{question}
Can one construct a counterexample to
Theorem~\ref{T:projective vs proper} over $k=\Fbar_p$?
\end{question}

\begin{example}
With notation as in Example~\ref{E:Raynaud},
one can show that the first projection $A \times A \times \tilde{B} \to A$
factors through a morphism $\calA \to A$
whose fiber above a given $a \in A(k)$ is projective if and only if
$a$ is a torsion point.
This provides a counterexample over $\Fbar_p$
except for the fact that $\calA \to A$ is not smooth.
\end{example}

\section{Comparison with the Hodge-theoretic approach}
\label{S:complex outline}

Both Andr\'e's Hodge-theoretic proof of Theorem~\ref{T:global}
and our proof
give information about the jumping locus
beyond the mere existence of a non-jumping point.
In this section, we compare these refinements.

\subsection{Refinement in terms of sparse sets}

To state the more precise information about the jumping locus
that Andr\'e's approach yields,
we introduce the notion of ``sparse'',
which is a version for closed points
of the notion ``thin'' defined in \cite{SerreMordellWeil1997}*{\S9.1}
in the context of Hilbert irreducibility.

\begin{definition}
\label{D:sparse}
Let $F$ be a finitely generated field over $\Q$.
Let $B$ be an $F$-variety.
Call a subset $S$ of $|B|$ \defi{sparse}
if there exists a dominant and generically finite morphism
$\pi \colon B' \to B$ of irreducible $F$-varieties
such that for each $s \in S$,
the fiber $\pi^{-1}(s)$ is either empty 
or contains {\em more} than one closed point.
\end{definition}

\begin{remark}
In the context of Definition~\ref{D:sparse},
call a subset of $B(\overline{F})$ \defi{sparse}
if and only if the associated set of closed points is sparse.
\end{remark}

Andr\'e's approach yields the following:

\begin{theorem}
\label{T:complex version}
Let $F$ be a field that is finitely generated over $\Q$.
Let $B$ be an irreducible $F$-variety,
and let $f \colon \calX \to B$ be a smooth projective morphism.
Then the set 
\[
	B_{\jumping} := 
	\{ b \in |B| : \rho(\calX_{\bar{b}}) > \rho(\calX_{\bar{\eta}}) \}.
\]
is sparse.
\end{theorem}

\begin{remark}
Although Theorem~\ref{T:complex version} has a restriction on $F$
and assumes that $f$ is projective (not just proper),
an easy argument
shows that it still implies the general case of Theorem~\ref{T:global}.
The extension to the proper case
uses Chow's lemma and weak factorization of birational maps
in characteristic~$0$ (see the first sentence of Remark~2
following Theorem~0.3.1 in~\cite{Abramovich-et-al2002}).
\end{remark}

The strategy of the proof of Theorem~\ref{T:complex version}
is to use Deligne's global invariant cycle theorem
(\cite{Deligne-HodgeII} or \cite{Voisin2003vol2}*{4.3.3})
and the semisimplicity of the category of polarized Hodge structures
to decompose the Hodge structure on
the Betti cohomology $\HH^2(\calX_b^\an,\Q)$
as $H \directsum H^\perp$,
where $H$ consists of the classes invariant
under some finite-index subgroup of the geometric monodromy group.
The space $H$ carries a Hodge structure independent of $b$;
after an \'{e}tale base change, $H$ comes from the cohomology
of a smooth completion of the total space of the family.
On the other hand, the ``essentially varying part'' $H^\perp$
is susceptible to 
Terasoma's argument 
(\cite{Terasoma1985} and \cite{Green-Griffiths-Paranjape2004}*{Lemma~8}), 
which shows that the 
corresponding subspace of \'etale cohomology has no nonzero Tate classes
for $b$ outside a sparse set.

For a complete proof, see \cite{Andre1996}*{\S5.2}.
For an exposition of the Hodge-theoretic aspects of the proof,
see also~\cite{Voisin-preprint}.

\subsection{Sparse versus nowhere $p$-adically dense}
\label{newsectionintro}

The purpose of this subsection is to develop basic properties of sparse sets,
and to discuss the difference between the two notions of smallness
appearing in Theorems \ref{T:complex version} and~\ref{T:local}.

\begin{proposition}
\label{P:sparse}
\hfill
  \begin{enumerate}[\upshape (a)]
  \item\label{I:sparse Zariski closed} For any dense open subscheme $U$ of $B$
and any $S \subseteq |B|$, the following are equivalent:
\begin{enumerate}
\item $S$ is sparse,
\item $S \intersect U$ is sparse as a subset of $|B|$,
\item $S \intersect U$ is sparse as a subset of $|U|$.
\end{enumerate}
In particular, for any closed subscheme $Z \subsetneq B$, 
the subset $|Z|$ is sparse in $|B|$.
  \item\label{I:sparse finite union} A finite union of sparse sets is sparse.
  \item\label{I:sparse image} A dominant and generically finite morphism of irreducible varieties maps sparse sets to sparse sets.
  \item\label{I:sparse thin} If $S$ is a sparse subset of $|\Aff^n|$, then $S \intersect \Aff^n(F)$ is a thin set.
  \item\label{I:sparse complement} If $S \subseteq |B|$ is sparse, then $|B|-S$ is Zariski dense in $B$.
  \end{enumerate}
\end{proposition}

\begin{proof}\hfill
  \begin{enumerate}[\upshape (a)]
  \item Take $B'=B-Z$.
  \item Let $V$ be an irreducible component dominating $B$
of the fiber product of the varieties $B'$ over $B$.
By (a), we may replace $B$ by a dense open subscheme $U$, 
and the other schemes by the corresponding preimages of $U$.
This lets us assume that each morphism $V \to B'$ and $B' \to B$ is surjective.
If $s$ is in the union of the sparse sets, then its preimage in some $B'$
contains more than one closed point, and so does its preimage in $V$.
  \item Let $f \colon C \to B$ be the morphism, 
and let $S$ be a sparse subset of $|C|$,
witnessed by $C' \to C$.  By (a), we may shrink $B$ to assume that
the morphisms $C' \to C \to B$ are surjective.  If $s \in S$,
then the preimage of $f(s)$ in $C'$ contains more than one closed point.
  \item By definition.
  \item If not, 
then by (a) there is a dense open subscheme $U$ of $B$ 
such that $|U|$ is sparse in $|U|$.
After shrinking $U$, there is a finite surjective morphism from $U$ to a
nonempty open subscheme $V$ of some $\Aff^n$.
Then $|V|$ is sparse in $|\Aff^n|$ by \eqref{I:sparse image},
so $V(F)$ is a thin set by \eqref{I:sparse thin}.
This contradicts the fact that a finitely generated field $F$ over $\Q$ 
is Hilbertian.
\qedhere
  \end{enumerate}
\end{proof}

The following proposition
shows that neither of the properties of being
sparse or nowhere $p$-adically dense implies the other.

\begin{proposition}
\label{P:neither implication}
\hfill
  \begin{enumerate}[\upshape (i)]
  \item
Let $F$ be a number field.  
Fix embeddings $F \subseteq \Qbar \subseteq \C_p$.  
Let $\pi \colon B' \to B$ be any dominant and generically finite morphism 
of irreducible $F$-varieties such that $\dim B>0$ and $\deg \pi > 1$.
Let $S$ be the preimage of
\[
	\{ b \in |B|: \pi^{-1}(b) \textup{ is not a single closed point}\}.
\]
in $B(\Qbar)$.
Then $S$ is dense in $B(\C_p)$.
  \item There exists a non-sparse subset $S$ of $|\Aff^1_\Q|$ such that
the associated subset of $\Aff^1(\Qbar)$ is nowhere $p$-adically dense
in $\Aff^1(\C_p)$.
  \end{enumerate}
\end{proposition}

\begin{proof}\hfill
  \begin{enumerate}[\upshape (i)]
  \item
Let $T:=\{b' \in B'(\Qbar) : F(b')=F(\pi(b'))\}$.
By replacing $B'$ and $B$ by dense open subschemes,
we may assume that $B' \to B$ is finite \'etale,
so $B'(\C_p) \to B(\C_p)$ is surjective.
Then it suffices to prove that $T$ is $p$-adically dense in $B'(\C_p)$,
or equivalently, $p$-adically dense in $B'(\Qbar)$.
To approximate a point $P$ in $B'(\Qbar)$ by a point in $T$, 
choose an irreducible $F$-curve $C$ in $B$ through $\pi(P)$
(maybe not geometrically irreducible),
and let $C'$ be the irreducible component of $\pi^{-1}(C)$ containing $P$.
By replacing $B' \to B$ by $C' \to C$, we may assume that $\dim B=1$.

After replacing $B'$ and $B$ by dense open subschemes,
we can compose with a finite \'etale morphism $B \to \Aff^1$
to reduce to the case that $B$ is a dense open subscheme of $\Aff^1$.
By the primitive element theorem, after shrinking again,
we may find a dominant morphism $\rho \colon B' \times \Aff^1$ such that 
$(\pi,\rho) \colon B' \to B \times \Aff^1$ is an immersion.

Let $u_0 \in B'(\Qbar)$.
We need to $p$-adically approximate $u_0$ by some $u \in T$.
Let $E:=F(u_0)$.
Let $d:=[E:F]$.
Let $\sigma_1,\ldots,\sigma_d$ be the $F$-embeddings $E \to \Qbar$.
Let $Z$ be the Zariski closure of the subscheme of $(\Aff^1)^d$ 
whose geometric points are
$(Q_1,\ldots,Q_d)$ such that for some $i \ne j$, 
we have $\pi(\rho^{-1}(Q_i)) \intersect \pi(\rho^{-1}(Q_j)) \ne \emptyset$.
Since $Z \ne (\Aff^1)^d$,
Lemma~\ref{L:Galois density} below gives $Q \in \Aff^1(E)$ 
$p$-adically close to $\rho(u_0)$
such that $(\sigma_1(Q),\ldots,\sigma_d(Q)) \notin Z$.
For sufficiently close $Q$, 
there exists $u \in \rho^{-1}(Q) \subseteq B'(\Qbar)$ 
$p$-adically close to $u_0$.
By definition of $Z$, the $\Gal(\Qbar/F)$-conjugates of $u$ 
map to distinct points in $B(\Qbar)$, so $u \in T$.
\item
The set $\Q$ is nowhere $p$-adically dense in $\C_p$.
On the other hand, $\Q$ is not thin in $\Q$,
so Proposition~\ref{P:sparse}\eqref{I:sparse thin}
proves that $\Q$ is not sparse.
  \end{enumerate}
\end{proof}

\begin{lemma}
\label{L:Galois density}
Let $E \supseteq F$ be an extension of number fields.
Fix inclusions $F \subseteq E \subseteq \Qbar \subseteq \C_p$.
Let $\sigma_1,\ldots,\sigma_d$ be the $F$-embeddings $E \to \Qbar$.
If $Z \subsetneq (\Aff^1_\Qbar)^d$ is a closed subscheme,
then the set 
$W:=\{Q \in \Aff^1(E): (\sigma_1(Q),\ldots,\sigma_d(Q)) \notin Z(\Qbar)\}$
is $p$-adically dense in $\Aff^1(E)$.
\end{lemma}

\begin{proof}
Let $R$ be the restriction of scalars $\Res_{E/F} \Aff^1$,
which is isomorphic to $\Aff^d_F$.
The bijection $\iota \colon R(F) \to \Aff^1(E)$ is $p$-adically continuous
for the topologies induced by $F \subseteq E \subseteq \C_p$
(though its inverse need not be).
The embeddings $\sigma_1,\ldots,\sigma_d$ define an isomorphism
$s \colon R_\Qbar \to (\Aff^1)^d_\Qbar$.
The complement of $s^{-1}(Z)$ in $R(F)$ is $p$-adically dense in $R(F)$,
so its image under $\iota$ is $p$-adically dense in $\Aff^1(E)$.
\end{proof}

\begin{remark}
\label{R:nonalgebraically closed field}
In the setting of Theorem~\ref{T:global},
but over a field $k$ of characteristic~$0$ that need not be
algebraically closed,
either Theorem~\ref{T:complex version} or Theorem~\ref{T:local}
can be used to show that there exists $d \ge 1$ such that
\[
	\left\{
	b \in B \mid [\kappa(b):k] \le d \text{ and }
			\rho(\calX_{\bar{b}}) = \rho(\calX_{\bar{\eta}})
	\right\}
\]
is Zariski dense in $B$.
In fact, Theorem~\ref{T:complex version} 
(with Proposition~\ref{P:sparse}(\ref{I:sparse image},\ref{I:sparse thin}))
proves the stronger result that the degree of any
generically finite rational map $B \dashrightarrow \PP^n$
can serve as $d$.
\end{remark}

%****************************************************************************
\section{Cycles of higher codimension}
\label{S:higher codimension}

Many of our arguments apply to specialization of cycles
of codimension greater than $1$, but the results are conditional.
Andr\'e obtains a generalization at the cost of replacing 
algebraic cycles by ``motivated cycles'', which are the same
if the Hodge conjecture is true.
In the $p$-adic approach,
we would need a higher-codimension analogue (Conjecture~\ref{C:Emerton}) 
of the $p$-adic Lefschetz $(1,1)$ theorem 
(Theorem~\ref{T:p-adic Lefschetz 1,1 theorem}).

\subsection{Cycle class maps}
\label{S:cycle class maps}

To state the generalizations, we recall some standard definitions.
For a smooth proper variety $X$ over a field,
let $\ZZ^r(X)_\Q$ be the vector space of
codimension-$r$ cycles with $\Q$ coefficients.

\begin{definition}
\label{D:cycle class map}
Let $X$ be a smooth proper variety over an algebraically closed field $k$
of characteristic~$0$.
Fix a prime $\ell$.
Define $\rho^r(X)$ as the dimension of the $\Q_\ell$-span of the image
of the \defi{$\ell$-adic cycle class map}
\[
	\cl_\et \colon \ZZ^r(X)_\Q \to \HH^{2r}_\et(X,\Q_\ell(r)).
\]
\end{definition}

The homomorphism $\cl_\et$ factors through the vector space
$\ZZ^r(X)_\Q/\alg$ of cycles modulo algebraic equivalence.
Suppose that $k'$ is an algebraically closed field extension of $k$.
Then every element of $\ZZ^r(X_{k'})_\Q/\alg$
comes from a cycle over $\kappa(V) \subseteq k'$
for some $k$-variety $V$,
and can be spread out over some dense open subvariety in $V$,
and hence is algebraically equivalent to the $k'$-base extension
of the resulting cycle above any $k$-point of $V$.
So $\ZZ^r(X)_\Q/\alg \to \ZZ^r(X_{k'})_\Q/\alg$ 
is surjective (it is injective too),
and $\rho^r(X)=\rho^r(X_{k'})$.
This, together with standard comparison theorems,
shows that $\rho^r(X)$ equals
its analogue defined using the \defi{Betti cycle class map}
\[
	\cl_{\Betti} \colon \ZZ^r(X)_\Q \to \HH^{2r}(X^\an,\Q)
\]
when $k=\C$,
and also its analogue defined using de Rham cohomology 
(see~\cite{Hartshorne1975}*{II.7.8}).
The comparison with Betti cohomology shows that
$\rho^r(X)$ equals $\dim_\Q \cl_{\Betti}(\ZZ^r(X)_\Q)$,
which is independent of~$\ell$.

As follows from \cite{SGA6}*{X~App~7,~especially~\S7.14},
all the facts in Section~\ref{S:specialization of NS}
(other than facts not needed for the proofs of Theorem~\ref{T:global},
namely the claims about ampleness and about the cokernel of specialization
being torsion-free)
have analogues with $\NS X$ replaced by $\cl_\et(\ZZ^r(X)_\Q)$.
In particular, in a smooth proper family,
$\rho^r(X)$ can only increase under specialization,
and the jumping locus for $\rho^r$
is a countable union of Zariski closed subsets.

\subsection{The $p$-adic approach}

Now assume that $K$, $\calO_K$, $k$ are as in Setup~\ref{A:STAR}.
Let $X \to \Spec \calO_K$ be a smooth proper morphism.
We have a specialization map $\spe$,
cycle class maps $\cl_{\dR}$ and $\cl_{\cris}$
(\cite{Hartshorne1975}*{II.7.8} and \cite{Gillet-Messing1987}, respectively),
and a comparison isomorphism
$\sigma_{\cris}$ (cf.~\cite{Berthelot1974}*{V.2.3.2})
making the following diagram commute:
\[
\xymatrix{
\ZZ^r(X_K)_\Q \ar[r]^{\spe} \ar[d]_{\cl_{\dR}} & \ZZ^r(X_k)_\Q \ar[d]^{\cl_{\cris}} \\
\HH^{2r}_{\dR}(X_K/K) \ar[r]^{\sigma_{\cris}} & \HH^{2r}_{\cris}(X_k/K). \\
}
\]

\begin{conjecture}[$p$-adic variational Hodge conjecture]
\label{C:Emerton}
Let notation be as above.
Let $Z_k \in \ZZ^r(X_k)_\Q$.
If $\sigma_{\cris}^{-1}(\cl_{\cris}(Z_k)) \in \Fil^r \HH^{2r}_{\dR}(X_K/K)$,
then there exists $Z \in \ZZ^r(X_K)_\Q$
such that $\cl_{\dR}(Z) = \sigma_{\cris}^{-1}(\cl_{\cris}(Z_k))$.
\end{conjecture}

\begin{remark}
Conjecture~\ref{C:Emerton} seems to be well-known to experts;
it is a variant of \cite{Emerton-preprint}*{Conjecture~2.2},
which is where we learned about it.
\end{remark}

\begin{remark}
The $r=1$ case of Conjecture~\ref{C:Emerton}
is a weak form of Theorem~\ref{T:p-adic Lefschetz 1,1 theorem}.
\end{remark}

Repeating the $p$-adic proof of Theorem~\ref{T:local} yields the following.

\begin{theorem}
\label{T:local in higher codimension}
Assume Conjecture~\ref{C:Emerton}.
Let notation be as in Setup~\ref{A:STAR}.
Let $r$ be a nonnegative integer.
Then the set
\[
	B(\calO_C)_{\jumping} := \{b \in B(\calO_C) :
	\rho^r(\calX_b) > \rho^r(\calX_{\bar{\eta}}) \}
\]
is nowhere dense in $B(\calO_C)$ for the analytic topology.
\end{theorem}

\subsection{The Hodge-theoretic approach}

To obtain a result using Andr\'e's approach
requires either replacing algebraic cycles with motivated cycles
as in \cite{Andre1996}*{\S5.2},
or else assuming the variational Hodge conjecture, which we now state:

\begin{conjecture}[Variational Hodge conjecture]
\label{variationalhodge}
Let $f \colon \calX\to B$ be a smooth projective morphism
between smooth quasi-projective complex varieties with $B$ irreducible,
and let $r$ be a nonnegative integer.
Let $b\in B(\C)$ and let $\alpha_b\in \cl_{\Betti}(\ZZ^r(\calX_b)_\Q)$.
If $\alpha_b$ is the restriction of a class $\alpha \in \HH^{2r}(\calX^\an,\Q)$
(or equivalently, by Deligne's global invariant cycle theorem, $\alpha_b$
is invariant under the monodromy action of $\pi_1(B^\an,b)$),
then there is a cycle class
$\alpha' \in  \cl_{\Betti}(\ZZ^r(\calX)_\Q)$
such that $\alpha_b$ is the restriction of $\alpha'$.
\end{conjecture}

\begin{remark} \label{remark26juillet}
The variational Hodge conjecture is a consequence of the Hodge conjecture
because Deligne's global invariant cycles theorem
(\cite{Deligne-HodgeII} or \cite{Voisin2003vol2}*{4.3.3}) 
shows that the class $\alpha_b$ above is the restriction of
a Hodge class $\overline{\alpha}$
on some smooth completion of $\calX$.
On the other hand, it is a priori a much weaker statement,
since the Hodge class $\overline{\alpha}$ is
an ``absolute Hodge class'' in the sense of \cite{Deligne-et-al1982}*{p.~28}.
 \end{remark}

Assuming Conjecture~\ref{variationalhodge},
Andr\'e's argument yields the following
extension of Theorem~\ref{T:complex version}:

\begin{theorem}
Assume that $k$ is an algebraic closure of
a field $k_f$ that is finitely generated over $\Q$.
Let $B$ be an irreducible $k_f$-variety.
Let $\calX \to B$ be a smooth proper morphism.
Let $r$ be a nonnegative integer.
If Conjecture~\ref{variationalhodge} holds for $r$,
then the set of $b\in B(k)$ such that
$\rho^r(\calX_b) > \rho^r(\calX_{\bar{\eta}})$ is sparse.
\end{theorem}

\begin{remark}
In fact, Andr\'e works with motivic Galois groups instead
of spaces of cycle classes, and hence proves more, 
namely that one can consider 
not only $\calX \to B$ but also all its fibered powers simultaneously.
\end{remark}

\begin{remark}

\end{remark}

\subsection{An implication}
\label{S:implication}

We introduced above two different variational forms of the Hodge conjecture,
in order to extend the results on N\'{e}ron-Severi groups
to higher-degree classes.
We now show that the $p$-adic version is stronger.
\begin{theorem}
\label{T:implication}
The $p$-adic variational Hodge conjecture
(Conjecture~\ref{C:Emerton})
implies the variational Hodge conjecture
(Conjecture~\ref{variationalhodge}).
\end{theorem}

The proof will use the following statement.

\begin{lemma}
\label{L:algebraic locus is countable union}
Let $K$ be any field of characteristic~$0$.
Let $B$ be a $K$-variety.
Let $f \colon \calX \to B$ be a smooth projective morphism.
Let $\alpha$ be a relative de Rham class, i.e.,
an element of $\HH^0(B,\R^{2r} f_* \Omega^\bullet_{\calX/B})$.
Then there exists a countable set of closed subschemes $V_i$ of $B$
such that for any field extension $L \supseteq K$
and any $b \in B(L)$, the restriction $\alpha_b \in \HH^{2r}_\dR(\calX_b/L)$
is algebraic (i.e., in $\cl_{\dR}(\ZZ^r(\calX_b)_\Q)$)
if and only if $b \in \Union_i V_i(L)$.
\end{lemma}

\begin{proof}
This is a standard consequence of the fact that
the codimension-$r$ subschemes in the fibers of $f$
are parametrized by countably many relative Hilbert schemes $H_j$,
each of which is projective over $B$.
Let $E$ be $\R^{2r} f_* \Omega^\bullet_{\calX/B}$ viewed as a geometric
vector bundle on $B$.
The relative cycle class map is a $B$-morphism $H_j \to \Fil^r E$;
its image $I_j$ is projective over $B$
and points on $I_j$ over an algebraically closed field
parametrize the classes of effective cycles given by $H_j$.
Taking $\Q$-linear combinations of these classes yields
countably many closed subvarieties $W_i$ of $\Fil^r E$
that together parametrize all de Rham cycle classes in fibers of $f$.
The desired subschemes $V_i$ are the inverse images of the $W_i$
under the section $B \to \Fil^r E$ given by $\alpha$.
\end{proof}

\begin{proof}[Proof of Theorem~\ref{T:implication}]
Let $\calX$, $B$, $f$, $b$, $\alpha_b$, and $\alpha$ be as in
Conjecture~\ref{variationalhodge}.
Let $Z \in \ZZ^r(\calX_b)_\Q$ be such that $\cl_{\Betti}(Z)=\alpha_b$.
Let $\overline{\calX}$ be a smooth projective completion of $\calX$.
As in Remark~\ref{remark26juillet}, we note that
$\alpha_b = \bar{\alpha}|_{\calX_b}$,
where $\bar{\alpha}\in \HH^{2r}(\overline{\calX}^\an,\Q)$ is a Hodge class.
In the Leray spectral sequence, $\alpha$ maps to a relative Betti class in
$\HH^0(B^\an,\RR^{2r} f_* \C)$,
and the relative Grothendieck comparison isomorphism identifies this with a relative de Rham class
$\alpha_{\dR} \in \HH^0(B,E)$,
where $E$ is the sheaf $\R^{2r} f_* \Omega^\bullet_{\calX/B}$
equipped with the Hodge filtration $\Fil^\bullet$ and Gauss-Manin connection.
Since $\bar{\alpha}$ is Hodge 
and restricts to the same section of $E$ as $\alpha$,
$\alpha_{\dR}$ is a horizontal section lying in $\HH^0(B,\Fil^r E)$.

There is a finitely generated subring $A \subset \C$
such that $\calX$, $B$, $f$, $b$, $Z$, $E$, and $\alpha_{\dR}$
are obtained via base change from corresponding objects over $A$,
which we next base extend by an embedding $A \injects \Z_p$
provided by \cite{CasselsLocalFields1986}*{Chapter~5,~Theorem~1.1}.
To ease notation, from now on we use the notation of Setup~\ref{A:STAR}
with $\calO_K:=\Z_p$,
and reuse the symbols above to denote the corresponding objects over $\calO_K$.
Let $s \in B(k)$ be the reduction of $b \in B(\calO_K)$.
We may assume that the constituents of $Z$ are flat over $\calO_K$,
so that the special fiber $Z_s$ is the specialization of $Z_K$.

As in the proof of Lemma~\ref{L:isocrystal output},
we use $\cl_{\cris}(Z_s) \in \HH^{2r}_{\cris}(\calX_s/K)$
to construct a horizontal section $\gamma_{\dR}(Z_s)$ of $E|_D$
for some closed polydisk neighborhood $D$ of $b$ inside the residue disk
in $B(\calO_K)$ corresponding to $s$.
(This time we need only the $\calO_K$-points,
not the $\calO_\Kbar$-points.)
Since $\gamma_{\dR}(Z_s)$ and $\alpha_{\dR,K}|_D$ are represented
by convergent power series on $D$,
it suffices to check that they are in equal in the formal completion
of $E$ at $b$.
The corresponding formal sections are horizontal 
and take the same value at $b$ (namely, $\cl_{\dR}(Z_K)$),
so they are equal.
In particular, for any $b' \in D \subset B(\calO_K)$,
the value of $\gamma_{\dR}(Z_s)$ at the generic point $b'_K$ of $b'$
is in $\Fil^r E_{b'_K}$.

So Conjecture~\ref{C:Emerton} applied to the $\calO_K$-morphism
$\calX_{b'} \to B_{b'}$
implies that 
there exists $Z' \in \ZZ^r(\calX_{b'_K})_\Q$
such that
\[
	\sigma_{\cris}(\cl_{\dR}(Z')) = \cl_{\cris}(Z_s).
\]
The isomorphism $\sigma_\cris$ also maps 
$\alpha_{\dR,K}|_{b'_K} = \gamma_{\dR}(Z_s)|_{b'_K}$
to $\cl_\cris(Z_s) \in \HH^2_\cris(\calX_s/K)$,
so
\[
	\cl_{\dR}(Z') = \alpha_{\dR,K}|_{b'_K}.
\]
So the value of $\alpha_{\dR,K}$ on any $K$-point in $D$ is an algebraic class.

Applying Lemma~\ref{L:algebraic locus is countable union} to
$\calX_K \to B_K$ and $\alpha_{\dR,K}$
yields a countable set of closed subschemes $V_i$ of $B_K$.
The previous paragraph shows that $D \subseteq \Union V_i(K)$.
By Lemma~\ref{L:countable union}, we have $\dim V_i = \dim B_K$ for some $i$.
This $V_i$ contains the generic point $\eta$ of $B_K$.
So $\alpha_{\dR,K}|_\eta=\cl_\dR(Y_\eta)$
for some $Y_\eta \in \ZZ^r(\calX_\eta)_\Q$.
Taking the closure in $\calX_K$ of the constituents of $Y_\eta$
defines some $Y_K \in \ZZ^r(\calX_K)_\Q$.
To $Y_K$ we associate a relative de Rham class
$\alpha'_{\dR,K} \in \HH^0(B_K,\R^{2r} f_* \Omega^\bullet_{\calX_K/B_K})$
and a Betti class $\alpha':=\cl_{\Betti}((Y_K)^\an) \in \HH^{2r}(\calX^\an,\Q)$,
where the complex analytic spaces are obtained by fixing
an $A$-algebra homomorphism $K \injects \C$.
The set $\{ q \in B_K : \alpha'_{\dR,K}|_q = \alpha_{\dR,K}|_q \}$
is closed and contains $\eta$, so it contains $b$,
which we now view as a $K$-point.
Under the comparison isomorphisms,
the equal de Rham classes $\alpha'_{\dR,K}|_b$ and $\alpha_{\dR,K}|_b$
correspond to the Betti classes $\alpha'_b$ and $\alpha_b$,
so $\alpha'_b = \alpha_b$.
\end{proof}

%****************************************************************************
\section*{Acknowledgements}

We thank Claire Voisin for explaining the complex-analytic approach
of Section~\ref{S:complex outline} to us, 
and for many discussions regarding other parts of the article.
We also thank
Ekaterina Amerik,
Yves Andr\'e,
Pete Clark,
Brian Conrad,
Daniel Bertrand,
Jordan Ellenberg,
Max Lieblich,
Ben Moonen,
Arthur Ogus, and
Richard Thomas,
for useful discussions and comments.

\end{document}